\documentclass[
10pt,																
a4paper,                           						
english,															
parskip=half,													
]{scrartcl}														

\usepackage{babel}											
\usepackage[latin1]{inputenc}								
\usepackage[T1]{fontenc}									
\usepackage{lmodern}											
\usepackage{csquotes}										

\usepackage[
style=alphabetic,												
sorting=nyt,													
maxnames=5,														
minnames=3,														
backend=bibtex8,												
]{biblatex}														
\bibliography{PaperStab}									

\usepackage{xcolor}											
\usepackage{graphicx} 										
\usepackage{picinpar}										
\usepackage{multicol}										
\usepackage[margin = 3cm]{geometry}						
\DeclareGraphicsExtensions{.jpg, .gif, .mps}			
\usepackage{tikz}												
\usetikzlibrary{arrows}										
\usetikzlibrary{decorations.pathmorphing}				

\usepackage{amsmath}											
\usepackage{amssymb}											
\usepackage{amsfonts}										
\usepackage[amsmath,thmmarks,hyperref]{ntheorem}	
\usepackage{dsfont}											
\usepackage{makeidx}    	          					
\usepackage[refpage]{nomencl}								
\usepackage{hyperref}										
\usepackage{mathtools}										
\numberwithin{equation}{section}							
\allowdisplaybreaks[3]										

\newcommand*{\N}{\mathbb{N}}

\newcommand*{\R}{\mathbb{R}}
\newcommand*{\C}{\mathbb{C}}
\renewcommand*{\epsilon}{\varepsilon}
\renewcommand*{\phi}{\varphi}
\renewcommand*{\kappa}{\varkappa}
\let\temptheta\theta											
\let\theta\vartheta											
\let\vartheta\temptheta										
\let\temprho\rho												
\let\rho\varrho												
\let\varrho\temprho											
\renewcommand*{\tilde}{\widetilde}
\renewcommand*{\hat}{\widehat}
\renewcommand*{\bar}{\overline}

\renewcommand*{\d}{\partial}

\newcommand*{\dH}{\, d\mathcal{H}}
\newcommand*{\skp}[2]{\left\langle #1,#2 \right\rangle}
\newcommand*{\norm}[1]{\left\|#1\right\|}
\newcommand*{\mint}{-\!\!\!\!\!\!\int}
\newcommand{\modulo}[2]{\mbox{\raisebox{0.4ex}{\ensuremath{#1}\hspace{0.1ex}}{\raisebox{-0.1ex}{\Large /}}\raisebox{-0.2ex}{\ensuremath{#2}}}}

\DeclareMathOperator{\spn}{span}

\DeclareMathOperator{\im}{Im}
\DeclareMathOperator{\Vol}{Vol}
\DeclareMathOperator{\id}{Id}

\theoremstyle{plain}											
\theoremheaderfont{\normalfont\bfseries}				
\theorembodyfont{\normalfont}								
\theoremnumbering{arabic}									
\theoremseparator{:}											
\theoremsymbol{\ensuremath{\square}}					
\newtheorem{defi}{Definition}[section]
\newtheorem{rem}[defi]{Remark}

\theoremstyle{plain}											
\theoremheaderfont{\normalfont\bfseries}				
\theorembodyfont{\normalfont\itshape}					
\theoremnumbering{arabic}									
\theoremseparator{:}											
\theoremsymbol{}												
\newtheorem{thm}[defi]{Theorem}
\newtheorem{lemma}[defi]{Lemma}

\theoremstyle{nonumberplain}								
\theoremheaderfont{\normalfont}							
\theorembodyfont{\normalfont}								
\theoremseparator{:}											
\theoremsymbol{\ensuremath{\blacksquare}}				
\newtheorem{proof}{\underline{Proof}}

\begin{document}

\begin{titlepage}
\title{Stability of Spherical Caps under the Volume-Preserving Mean Curvature Flow with Line Tension}
\author{Helmut Abels\footnote{Fakult\"at f\"ur Mathematik,  
Universit\"at Regensburg,
93040 Regensburg,
Germany, e-mail: {\sf helmut.abels@mathematik.uni-regensburg.de}}\ \ 
Harald Garcke\footnote{Fakult\"at f\"ur Mathematik,  
Universit\"at Regensburg,
93040 Regensburg,
Germany, e-mail: {\sf harald.garcke@mathematik.uni-regensburg.de}}\ \
and Lars M\"uller\footnote{Fakult\"at f\"ur Mathematik,  
Universit\"at Regensburg,
93040 Regensburg,
Germany}}
\date{}
\end{titlepage}
\maketitle

\begin{abstract}
We show stability of spherical caps (SCs) lying on a flat surface, where the motion is governed by the volume-preserving Mean Curvature Flow (MCF). Moreover, we introduce a dynamic boundary condition that models a line tension effect on the boundary. The proof is based on the \textit{generalized principle of linearized stability}.
\end{abstract}

\small \textbf{Keywords:} mean curvature flow, stability, dynamic boundary conditions, line energy, spherical caps

\small \textbf{AMS subject classifcations:} 53C44, 35K55, 35B35, 37L15

\normalsize

\section{Introduction}\label{sec:Introduction}

The geometric evolution law $V_\Gamma = H_\Gamma$, meaning that the motion of a point on the surface in normal direction $V_\Gamma$ is equal to the mean curvature of the surface in that point, has many applications in geometry, physics and materials science. For example the evolution of grain boundaries is governed by mean curvature flow. First important results by mean curvature flow are due to Brakke \cite{Bra78}, Gage and Hamilton \cite{GH86} and Huisken \cite{Hui84}. The flow $V_\Gamma = H_\Gamma$ is known as the mean curvature flow (MCF) and with the additional condition of volume conservation, this flow appears e.g. as a model for surface attachment limited kinetics (SALK), see e.g. Cahn and Taylor \cite{CT94}. In 1987 it was Huisken \cite{Hui87} and in 1998 Escher and Simonett \cite{ES98b}, who provided important results concerning the volume-preserving MCF.
Volume preserving mean curvature flow of rotationally symmetric surfaces
with boundary contact has been studied by Athanassenas \cite{Ath98}, see
also the recent work \cite{AthKan12}. Stability of cylinders
under volume preserving mean curvature flow with
a $90$-degree angle condition at an external boundary 
has been studied  by Hartley \cite{Hart13}.

This paper is devoted to stability of spherical caps in $\R^3$ that
lie on a flat surface $\R^2 \times \{0\}$. Modelling a drop of liquid
or a soap bubble physics suggest that the air-liquid-interface, which
can be viewed as an evolving hypersurface, tends to minimize its
area. If such a surface gets into contact with some fixed impermeable
boundary layer the mass conservation law makes it necessary to demand
a constant volume condition. The occurring contact angle is mainly
determined by the material constants and thereby the wettability of
the container. The free energy is given as
\begin{align*}
\mathcal{E}(\Gamma):=\int_\Gamma 1 \dH^2 - \int_D a \dH^2
\end{align*}
where $\dH^d$, $d \in \{1, 2\}$ denotes integration with respect to
the $d$-dimensional Hausdorff measure, $a>0$ is a constant and $D$ is
the wetted region. The first term measures surface energy and the
second term accounts for contact energy. Then the angle
$\alpha$ at the junction line is determined by $\cos\alpha = -a$, see
Figure~\ref{fig:SphericalCaps} and Finn \cite{finn}. We remark that
the contact angle, which is typically used in physics, is given as
 $\gamma=\pi-\alpha$. However,
in particular on small length scales, a second
effect is entering the scenery, namely the line tension (cf. Section 1
of \cite{BLK06}). This effect penalizes long contact curves and forces
the drop or bubble to detach more from the boundary. The governing
energy for a hypersurface $\Gamma \subseteq \R^3$ with contact to
$\R^2 \times \{0\}$ is in this case given as
\begin{align*}
  \mathcal{F}(\Gamma) := \int_\Gamma 1 \dH^2- \int_D
  a \dH^1 + \int_{\d \Gamma} b \dH^1,
\end{align*}
where $b > 0$ is a constant. The last term accounts for line energy
effects. For a mathematical treatment of variational problems related
to $\mathcal{F}$ we refer to Morgan \cite{Mor94a,Mor94b},
Morgan and Taylor \cite{MT91} and Cook \cite{Coo85}. The motion of
such an evolving hypersurface $\Gamma$, which is schematically
illustrated in Figure \ref{fig:GammaInOmega}, will be a suitable
gradient flow of the energy $\mathcal{F}$.

\begin{figure}[htbp]
	\centering
	\begin{tikzpicture}[scale=1,>=stealth]
		\draw[shift={(-4,0)}] (-1.7,0) .. controls (-3,2) and (-1,4) .. (0,3) node[above=8pt,right=0pt] {$\Gamma(t)$};
		\draw[shift={(-4,0)}] (0,3) .. controls (1,2) and (3.5,1) .. (2.5,0);
		\draw[shift={(-4,0)}] (-3,0) node[below=9pt,right=12pt] {$\d \Omega$} -- (3,0);
		
		\draw [->,decorate,decoration={snake,amplitude=0.4mm,segment length=2mm,post length=1mm}] (-1,1.5) -- node[above] {$t \longrightarrow T$} (1,1.5);
		
		\draw[shift={(4,0)}] (2,0) arc (-23.5:203.5:2.3) node[above=76pt,right=59pt] {$\Gamma(T)$};
		\draw[shift={(4,0)}] (-3,0) node[below=9pt,right=12pt] {$\d \Omega$} -- (2.5,0);
	\end{tikzpicture}
	\caption{Evolving hypersurface $\Gamma(t)$ in contact with a container boundary $\d \Omega$}
	\label{fig:GammaInOmega}
\end{figure}
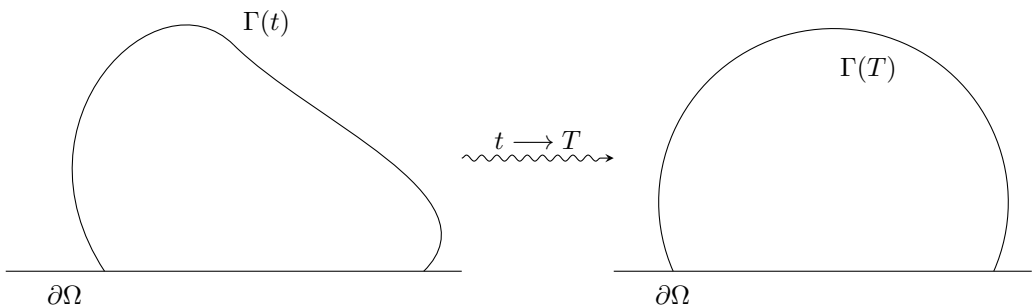

During this motion it seems artificial to prescribe the boundary curve or the contact angle, since an arbitrary drop or bubble, which is brought in contact with a solid container, will not instantly have a boundary curve or contact angle that is energetically minimal. Prescribing the contact curve or the contact angle would correspond to Dirichlet or Neumann boundary conditions, respectively. Instead of doing so, we will impose dynamic boundary conditions to allow the contact angle to change and the boundary curve to move. We will prove stability for spherical caps, which are the simplest stationary surfaces of the given flow.
It will turn out that the set of equilibria forms a three-dimensional manifold. This is due to the fact that we have two degrees of freedom with respect to horizontal translations and another degree of freedom stems from a change in the enclosed volume. As a consequence the classical theory of linearized stability does not apply and we have to use the generalized principle of linearized stability as introduced by Pr\"uss, Simonett and Zacher in \cite{PSZ09}.

After some elementary results on spherical caps in Section \ref{sec:SphericalCaps}, we will introduce in Section \ref{sec:GPLS} the generalized principle of linearized stability, which is the basis of out stability analysis. We will also introduce the abstract setting concerning the involved operators and spaces. Before we can apply the principle in Section \ref{sec:Application} by checking the four assumptions that are needed and formulate our final stability result in Theorem \ref{thm:StabilitySCs}, we need some perturbation result from semigroup theory to deal with the non-locality of the volume-preserving MCF in Section \ref{sec:MaxReg}. In order to show stability of stationary solutions we in particular need to study the spectrum of the surface Laplacian on the spherical cap with non-standard boundary conditions.

\section{Spherical Caps}\label{sec:SphericalCaps}

We want to consider the motion of an evolving hypersurface $\Gamma = (\Gamma(t))_{t \in I}$ inside the upper half space $\Omega := \R^3_+ := \{(x,y,z) \in \R^3 \mid z > 0\}$, which remains in contact with the boundary $\d \Omega$ given as the $x$-$y$-plane. With $V \subseteq \Omega$ we want to denote the region between $\Gamma$ and $\d \Omega$ and $D$ shall be defined as $D := \d V \cap \d \Omega$. In particular, we have $\d D = \d \Gamma$. For a point $p \in \Gamma$ we denote the exterior normal to $\Gamma$ in $p$ by $n_\Gamma(p)$, where the term ``exterior'' should be understood with respect to $V$. Obviously, the normal $n_D$ of $V$ on $D$ is the negative of the third unit vector. Furthermore, for a point $p \in \d \Gamma$ we want to denote by $n_{\d \Gamma}$ and $n_{\d D}$ the outer conormals to $\d \Gamma$ and $\d D$ in $p$. In addition, we define the tangent vector to the curve $\d \Gamma$ by $\vec{\tau}(p) := \frac{c'(t)}{|c'(t)|}$ and its curvature vector by $\vec{\kappa}(p) := \frac{1}{|c'(t)|} \left(\frac{c'(t)}{|c'(t)|}\right)'$, where $c: (t - \epsilon, t + \epsilon) \longrightarrow \d \Gamma$ is a parametrization of $\d \Gamma$ around $p \in \d \Gamma$ with $c(t) = p$.

For two parameters $a \in \R$ and $b > 0$ the motion of $\Gamma$ shall be driven by the volume-preserving mean curvature flow with a dynamic boundary condition
\begin{align}\label{eq:SimFlow1}
	V_\Gamma(t) &= H_\Gamma(t) - \bar{H}(t), \\ \label{eq:SimFlow2}
	v_{\d D}(t) &= a + b \kappa_{\d D}(t) + \cos(\alpha(t)).
\end{align}
Here $V_\Gamma$ is the normal velocity, $H_\Gamma$ is the mean curvature given as the sum of the principle curvatures and $\bar{H}(\rho(t))$ is the mean value of the mean curvature, defined as
\begin{align*}
	\bar{H}(t) := \mint_{\Gamma(t)} H_{\Gamma(t)}(t,p) \dH^2 := \frac{1}{\int_{\Gamma(t)}\limits 1 \dH^2} \int_{\Gamma(t)} H_{\Gamma(t)}(t,p) \dH^2.
\end{align*}
The term $\bar{H}(t)$ is exactly the right choice to make this flow volume-preserving as we can see by calculating the first variation of the volume
\begin{align*}
	\frac{d}{dt} \Vol(\Gamma(t)) = \int_{\Gamma(t)} V_{\Gamma(t)} \dH^2 = \int_{\Gamma(t)}(H_{\Gamma(t)} - \bar{H}) \dH^2 = \int_{\Gamma(t)}  H_{\Gamma(t)} \dH^2 - \bar{H} \int_{\Gamma(t)} 1 \dH^2 = 0.
\end{align*}
Moreover, $v_{\d D}$ is the normal boundary velocity of the contact curve, $\kappa_{\d D}$ is its geodesic curvature with respect to $\d \Omega$ and $\alpha$ is the contact angle of $\Gamma$ and $D$. We assume throughout the whole paper
\begin{align}\label{eq:AngleAssumption}
	0 < \alpha(p) < \pi \qquad \text{for all } p \in \d \Gamma,
\end{align}
which will be crucial later on.

Stationary hypersurfaces of (\ref{eq:SimFlow1})-(\ref{eq:SimFlow2}) have to satisfy
\begin{align}\label{eq:NecessrayConditions1}
	0 &= H_\Gamma - \bar{H} & &\text{ in } \Gamma, \\ \label{eq:NecessrayConditions2}
	0 &= a + b \kappa_{\d D} + \skp{n_\Gamma}{n_D} & &\text{ on } \d \Gamma.
\end{align}
Looking at the first equation we see that spherical caps - which we will call SCs hereafter - satisfy this equation. This motivates our aim to investigate SCs in this paper.

\begin{figure}[htbp]
	\centering
	\begin{tikzpicture}[scale=0.96,>=stealth]
		\filldraw[shift={(-4,0)},fill=green!20!white, draw=green!50!black] (2.3,0) -- (2.3,-0.7) arc (270:336.5:0.7) -- cycle;
		\draw[shift={(-4,0)},thin,white] (2.3,0) circle (0pt) node[below=10pt,right=0pt,color=green!50!black] {$\alpha$};
		\filldraw[shift={(-4,0)},fill=green!20!white, draw=green!50!black] (0,1) -- (0,0.3) arc (270:336.5:0.7) -- cycle;
		\draw[shift={(-4,0)},thin,white] (0,1) circle (0pt) node[below=10pt,right=0pt,color=green!50!black] {$\alpha$};
		\draw[shift={(-4,0)}] (0,1) -- node[above] {$R$} +(-23.5:2.5) arc (-23.5:203.5:2.5) node[above=50pt,left=3pt] {$\Gamma$};
		\draw[shift={(-4,0)}] (-2.3,0) node[below=7pt,right=10pt] {$\d \Omega$} -- (2.3,0);
		\draw[shift={(-4,0)}] (0,0) -- node[left] {$H$} (0,1);
		\draw[shift={(-4,0)}] (0,0) -- node[below] {$r$} (2.3,0);
		\fill[shift={(-4,0)},thick] (0,1) circle (2pt) node[above=25pt,left=7pt] {$V$};
		\draw[shift={(-4,0)},thick,->] (2.3,0) -- +(1.5,0) node[above] {$n_{\d D}$};
		\draw[shift={(-4,0)},thick,->] (2.3,0) -- +(0,-1.5) node[below] {$n_D$};
		\draw[shift={(-4,0)},thick,->] (2.3,0) -- +(-113.5:1.5) node[left] {$n_{\d \Gamma}$};
		\draw[shift={(-4,0)},thick,->] (2.3,0) -- +(-23.5:1.5) node[below] {$n_\Gamma$};
		
		\filldraw[shift={(3.5,0)},fill=green!20!white, draw=green!50!black] (2.3,0) -- (2.3,-1) arc (270:383.5:1) -- cycle;
		\draw[shift={(3.5,0)},thin,white] (2.3,0) circle (0pt) node[below=10pt,right=8pt,color=green!50!black] {$\alpha$};
		\draw[shift={(3.5,0)}] (0,-1) -- node[below] {$R$} +(23.5:2.5) arc (23.5:156.5:2.5) node[above=23pt,right=0pt] {$\Gamma$};
		\draw[shift={(3.5,0)}] (-2.3,0) node[below=7pt,right=10pt] {$\d \Omega$} -- (2.3,0);
		\draw[shift={(3.5,0)}] (0,0) -- node[left] {$H$} (0,-1);
		\draw[shift={(3.5,0)}] (0,0) -- node[above] {$r$} (2.3,0);
		\fill[shift={(3.5,0)},thick] (0,-1) circle (2pt) node[above=45pt,left=7pt] {$V$};
		\draw[shift={(3.5,0)},thick,->] (2.3,0) -- +(1.5,0) node[right] {$n_{\d D}$};
		\draw[shift={(3.5,0)},thick,->] (2.3,0) -- +(0,-1.5) node[below] {$n_D$};
		\draw[shift={(3.5,0)},thick,->] (2.3,0) -- +(-66.5:1.5) node[right] {$n_{\d \Gamma}$};
		\draw[shift={(3.5,0)},thick,->] (2.3,0) -- +(23.5:1.5) node[right] {$n_\Gamma$};
	\end{tikzpicture}
	\caption{Notation for spherical caps}
	\label{fig:SphericalCaps}
\end{figure}
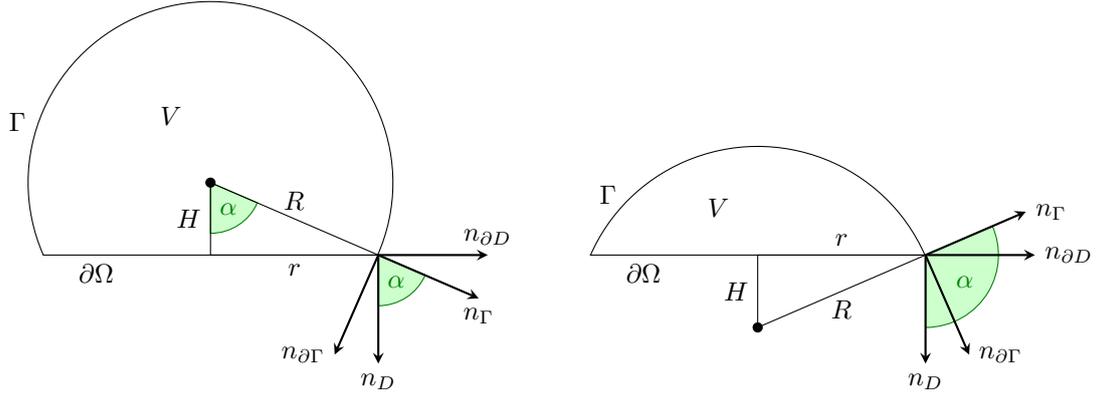

The radius of the SC shall be denoted by $R$ and the height of its center by $H$. Our convention will be that an SC whose center is above $\d \Omega$ has a positive $H$ and if the center is below the $x$-$y$-plane we declare $H$ to be negative. The contact curve $\d \Gamma = \d D$ in this case is obviously an ordinary circle whose radius will be denoted by $r$. For a sketch of this notation see Figure \ref{fig:SphericalCaps}. Note that $\alpha$ is constant in this situation. Our sign convention for $H$ leads to
\begin{align}\label{eq:RAndH}
	H = r \cot(\alpha) \quad \text{ and } \quad R = \frac{r}{\sin(\alpha)}.
\end{align}
The triple $(\vec{\tau}, n_D, n_{\d D})$ is supposed to be a right-handed orthonormal basis, hence we have to parametrize the contact circle clockwise looking down from the north pole. This causes the arc length derivative of $\vec{\tau}$, which is the curvature vector $\vec{\kappa}$, to point inwards and away from $n_{\d D}$. Therefore the geodesic curvature of the contact curve is negative, which means $\kappa_{\d D} = -\frac{1}{r}$.

An SC is a stationary SC - which we denote by SSC - if it satisfies (\ref{eq:NecessrayConditions2}), which simplifies to
\begin{align}\label{eq:StationaryAngle}
	\cos(\alpha) = \frac{b}{r} - a.
\end{align}
Looking at (\ref{eq:StationaryAngle}) we immediately see that $-1 < \frac{b}{r} - a < 1$ has to hold, where we exclude the cases $\cos(\alpha) = \pm 1$, because they correspond to the two degenerate cases of a SC that has fully detached from $\d \Omega$ or has completely spread out to become flat. We can therefore distinguish the following cases:
\begin{enumerate}
	\item Case ($a \leq -1$): Here we should have $a - 1 < \frac{b}{r} < a + 1 \leq 0$, which is not possible since $b > 0$ and $r > 0$.
	\item Case ($-1 < a \leq 1$): Here the left inequality of $a - 1 < \frac{b}{r} < a + 1$ is always satisfied and we have to ensure $r \in \left(\frac{b}{a+1}, \infty\right)$.
	\item Case ($a > 1$): Now both inequalities restrict $r$ and we obtain $r \in \left(\frac{b}{a+1}, \frac{b}{a-1}\right)$.
\end{enumerate}
This shows that there are definitely no SSCs if $a \leq -1$ and hence in the following considerations we assume
\begin{align}\label{eq:ParameterEquations}
	a > -1 \qquad \text{ and } \qquad b > 0.
\end{align}
The range that $r$ can attain is given by
\begin{align}\label{eq:RangeOfR}
	I_r := \begin{cases}
				 \left(\frac{b}{a+1}, \infty\right)        & \text{ if } -1 < a \leq 1, \\
				 \left(\frac{b}{a+1}, \frac{b}{a-1}\right) & \text{ if } a > 1.
			 \end{cases}
\end{align}
The term $\cos(\alpha) = \frac{b}{r} - a$ is obviously strictly decreasing in $r$. Thus
\begin{align*}
	\cos(\alpha) = \frac{b}{r} - a \uparrow 1 \qquad \text{ for} \quad r \downarrow \frac{b}{a+1}
\end{align*}
and in case $a > 1$ we furthermore have
\begin{align*}
	\cos(\alpha) = \frac{b}{r} - a \downarrow -1 \qquad \text{ for} \quad r \uparrow \frac{b}{a-1},
\end{align*}
which shows that all contact angles $\alpha \in (0,\pi)$ are possible. \\
Looking at the case $-1 < a \leq 1$ we obtain the limit
\begin{align*}
	\cos(\alpha) = \frac{b}{r} - a \downarrow -a \qquad \text{ for } r \longrightarrow \infty
\end{align*}
and therefore only $\alpha \in (0,\arccos(-a))$ can appear as contact angle of an SSC. So we obtain
\begin{align}\label{eq:RangeOfAlpha}
	I_\alpha := \begin{cases}
						\left(0, \arccos(-a)\right) & \text{ if } -1 < a \leq 1, \\
						\left(0, \pi\right)			 & \text{ if } a > 1
					\end{cases}
\end{align}
as the feasible range for $\alpha$.

\begin{figure}
	\centering
	\begin{tikzpicture}
		\draw[very thin] (-6:1.5) arc (-6:195:1.5);
		\draw[very thin] (2:2) arc (2:186:2);
		\draw[very thin] (13.5:3) arc (13.5:175:3);
		\draw[very thin] (18:3.5) arc (18:170:3.5);
		\draw[] (0,1.67) --  node[right=15pt,below=-8pt,fill=white]{$\rho(t,q)$} (0,2.5);
		\draw[] (0,1.67) --  (0,2.5);
		\draw[very thick] (8:2.5) arc (8:180:2.5) node[below=2pt] {$\Gamma^*$};
		\draw[thick] (3.05,0.8) .. controls (2,2) and (-3,2) .. (-3.3,0.47) node[below=10pt,left=-5pt] {$\Gamma_\rho(t)$};
		\fill[thick] (0,2.5) node[above] {$q$} circle (2pt);
		\draw[thick] (4,2) arc (-30:-150:5) node[left] {$\d \Omega$};
		\draw[shift={(0.1,-0.5)},<->] (3.35,1.3) arc (-44:-70:5) node[above=5pt,right=25pt] {$w$};
	\end{tikzpicture}
	\caption{The distance function $\rho$}
	\label{fig:Rho}
\end{figure}
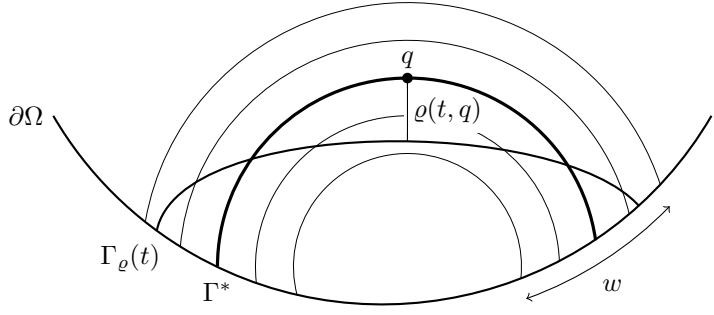

Our next goal is to perform a Hanzawa transformation and write the evolving hypersurface as a family of graphs of a time-dependent ``distance-like'' function $\rho: [0,T] \times \Gamma^* \longrightarrow (-\epsilon_0, \epsilon_0)$ over a fixed reference hypersurface $\Gamma^*$, which we assume to be an SSC. The distance $\rho(t,q)$ of a point $q \in \Gamma^*$ shall be measured in normal direction as indicated in Figure \ref{fig:Rho}. But this is not possible for a boundary point $q \in \d \Gamma^*$. In our situation we need some correction term to ensure that the evolving hypersurface $\Gamma$ neither crosses $\d \Omega$ nor detaches from it.

For this purpose we introduce a curvilinear coordinate system $\Psi$
as introduced by Vogel \cite{Vog00}, see also \cite{Depner, DepnerGarcke}, because with its help we can write an evolving hypersurface as a graph over the fixed reference SSC $\Gamma^*$.

For $q \in \d \Gamma^*$ and $w \in (-\epsilon_0, \epsilon_0)$ with $\epsilon_0 > 0$ sufficiently small there is a smooth function
\begin{align*}
	\tilde{t}: \d \Gamma^* \times (-\epsilon_0, \epsilon_0) \longrightarrow \R: (q,w) \longmapsto \tilde{t}(q,w)
\end{align*}
such that
\begin{align*}
	q + w n_{\Gamma^*}(q) + \tilde{t}(q,w) n_{\d \Gamma^*}(q) \in \d \Omega \qquad \forall \ w \in (-\epsilon_0, \epsilon_0).
\end{align*}
Obviously, $\tilde{t}(q,0) = 0$ and we can extend $\tilde{t}$ smoothly to a function
\begin{align*}
	t: \Gamma^* \times (-\epsilon_0, \epsilon_0) \longrightarrow \R: (q,w) \longmapsto t(q,w)
\end{align*}
such that $t(q,0) = 0$ for all $q \in \Gamma^*$. Next we will use a special coordinate system
\begin{align}\label{eq:CurvilinearCoordinates}
	\Psi: \Gamma^* \times (-\epsilon_0, \epsilon_0) \longrightarrow \Omega: (q,w) \longmapsto \Psi(q,w) := q + w n_{\Gamma^*}(q) + t(q,w) T(q),
\end{align}
where $T: \Gamma^* \longrightarrow \R^3$ is a tangential vector field, that coincides with $n_{\d \Gamma^*}$ on $\d \Gamma^*$ and vanishes outside a small neighborhood of $\d \Gamma^*$. By construction this curvilinear coordinate system satisfies $\Psi(q,0) = q$ for all $q \in \Gamma^*$ and $\Psi(q,w) \in \d \Omega$ for all $q \in \d \Gamma^*$ and all $w \in (-\epsilon_0, \epsilon_0)$. The existence of such a curvilinear coordinate system is guaranteed due to (\ref{eq:AngleAssumption}) which is a result from \cite{Vog00}, where one can also find more technical details concerning $\Psi$.

We define our evolving hypersurface $\Gamma := (\Gamma_\rho(t))_{t \in I}$ via $\Gamma_\rho(t) := \im(\Psi(\bullet,\rho(t,\bullet)))$ and observe that by our construction of $\Psi$ we have $\Gamma_0(t) = \Gamma^*$ for all $t \in [0,\infty)$. We assume that $\rho$ is smooth enough such that all the upcoming terms are defined.

The precise flow that we want to consider is
\begin{align}\label{eq:Flow1}
	V_\Gamma(\Psi(q,\rho(t,q))) &= H_\Gamma(\Psi(q,\rho(t,q))) - \bar{H}(t) & &\text{in } \Gamma^*, \\ \label{eq:Flow2}
	v_{\d D}(\Psi(q,\rho(t,q))) &= a + b \kappa_{\d D}(\Psi(q,\rho(t,q))) & & \notag \\
										 &+ \skp{n_\Gamma(\Psi(q,\rho(t,q)))}{n_D(\Psi(q,\rho(t,q)))} & &\text{on } \d \Gamma^*.
\end{align}

For later purposes the linearization of (\ref{eq:Flow1})-(\ref{eq:Flow2}) will be crucial. The calculations leading to the linearization given by
\begin{align}\label{eq:LinearFlow1}
	\d_t \rho(t) & = \Delta_{\Gamma^*} \rho(t) + |\sigma^*|^2 \rho(t) - \mint_{\Gamma^*} (\Delta_{\Gamma^*} + |\sigma^*|^2) \rho(t) \dH^2 \hspace*{19mm} \text{ in } [0,T] \times \Gamma^*, \\ \label{eq:LinearFlow2}
	\d_t \rho(t) & = -\sin(\alpha^*)^2 (n_{\d \Gamma^*} \cdot \nabla_{\Gamma^*} \rho(t)) + \sin(\alpha^*) \cos(\alpha^*) II_{\Gamma^*}(n_{\d \Gamma^*},n_{\d \Gamma^*}) \rho(t) \notag \\ 
					 & + b \sin(\alpha^*) \rho_{\sigma\sigma}(t) - b \sin(\alpha^*) \kappa_{\d D^*} \skp{\vec{\tau}^*}{(n_{\d D^*})_\sigma} \rho(t) \hspace*{17mm} \text{ on } [0,T] \times \d \Gamma^*
\end{align}
can be found in Section 2 of \cite{Mue13}, see also
\cite{AGMPreprint1,Depner, DepnerGarcke}.

After we know which conditions have to hold for the contact angle $\alpha$ and the radius $r$ and how we describe the motion of the hypersurface we can now start with the stability analysis of SCs.

\section{The Generalized Principle of Linearized Stability}\label{sec:GPLS}

Since we assumed that the reference hypersurface $\Gamma^*$ is an SSC our goal is to prove the stability of the zero-solution $\rho \equiv 0$ for (\ref{eq:Flow1})-(\ref{eq:Flow2}). To this end we will use the generalized principle of linearized stability (GPLS) as presented in \cite{PSZ09} and start by introducing the abstract framework.

We begin by transforming the equations (\ref{eq:Flow1})-(\ref{eq:Flow2}) into an abstract evolution equation of the form
\begin{align}\label{eq:AbstractEvolutionEquation1}
	\d_t v(t) + A(v(t)) v(t) &= F(v(t)) \qquad t \in \R_+, \\ \label{eq:AbstractEvolutionEquation2}
							  v(0) &= v_0
\end{align}
as given by (2.1) in \cite{PSZ09}. As in Lemma 2.10 of \cite{Mue13} we can extract $\d_t \rho$ from $V_\Gamma$ and transform (\ref{eq:Flow1}) into
\begin{align*}
	\d_t \rho(t,q) &= \frac{H_\Gamma(\rho(t,q)) - \bar{H}(\rho(t))}{n_\Gamma(\rho(t,q)) \cdot \d_w \Psi(q,\rho(t,q))} \qquad \text{ in } \R_+ \times \Gamma^*.
\end{align*}
Analogously we transform (\ref{eq:Flow2}) into
\begin{align*}
	\d_t \rho(t,q) = \frac{a + b \kappa_{\d D}(\rho(t,q)) + \skp{n_\Gamma(\rho(t,q))}{n_D(\rho(t,q))}}{n_{\d D}(\rho(t,q)) \cdot \d_w \Psi(q,\rho(t,q))} \qquad \text{ on } \R_+ \times \d \Gamma^*.
\end{align*}
For $4 < p < \infty$ we define
\begin{align*}
	X_1 & := \mathcal{D}(A) := \left\{\left.(u,\varrho) \in W^2_p(\Gamma^*;\R) \times W^{3-\frac{1}{p}}_p(\d \Gamma^*;\R) \right| u|_{\d \Gamma^*} = \varrho\right\}, \\
	X_0 & := W := L_p(\Gamma^*;\R) \times W^{1-\frac{1}{p}}_p(\d \Gamma^*;\R),
\end{align*}
where $X_1 \hookrightarrow X_0$ as demanded in \cite{PSZ09}. By
interpolation results as in Theorem 4.3.1/1 of \cite{Tri78}, which also
hold on surfaces, we obtain
\begin{align*}
	\left(L_p(\Gamma^*), W^2_p(\Gamma^*)\right)_{1-\frac{1}{p},p} & = W^{2-\frac{2}{p}}_p(\Gamma^*), \\
	\left(W^{1-\frac{1}{p}}_p(\d \Gamma^*), W^{3-\frac{1}{p}}_p(\d \Gamma^*)\right)_{1-\frac{1}{p},p} & = W^{3-\frac{3}{p}}_p(\d \Gamma^*).
\end{align*}
Corollary 1.14 of \cite{Lun09} shows that functions $(u,\varrho)$ belonging to $\left(X_0, X_1\right)_{1-\frac{1}{p},p}$ are traces at $t = 0$ of functions $v \in W^1_p(\R_+;X_0) \cap L_p(\R_+;X_1) \hookrightarrow BUC([0,\infty);\left(X_0, X_1\right)_{1-\frac{1}{p},p})$. This proves that the trace condition $u|_{\d \Gamma^*} = \varrho$ carries over from $X_1$ to the interpolation space and we have 
\begin{align*}
	X_\gamma := \left(X_0, X_1\right)_{1-\frac{1}{p},p} \subseteq \left\{\left.(u,\varrho) \in W^{2-\frac{2}{p}}_p(\Gamma^*) \times W^{3-\frac{3}{p}}_p(\d \Gamma^*) \right| u|_{\d \Gamma^*} = \varrho\right\}.
\end{align*}
Moreover, calculating the mean curvature with respect to the used coordinates one observes that
\begin{align*}
	A_1(u,\varrho) (u,\varrho)	& := -\frac{H_\Gamma(u(t,q))}{n_\Gamma(u(t,q)) \cdot \d_w \Psi(q,u(t,q))}, \\
	A_2(u,\varrho) (u,\varrho) & := -\frac{a + b \kappa_{\d D}(\varrho(t,q)) + \skp{n_\Gamma(u(t,q))}{n_D(u(t,q))}}{n_{\d D}(\varrho(t,q)) \cdot \d_w \Psi(q,\varrho(t,q))}
\end{align*}
are quasilinear differential operators. More precisely, one can show that there are  
 $V := B_\epsilon(0) \subseteq X_\gamma$, $\epsilon>0$, and $A \in C^1(V, \mathcal{L}(X_1,X_0))$ such that
 \begin{equation*}
	A(v)v := \begin{pmatrix}
					 A_1(v)v \\
					 A_2(v)v
				 \end{pmatrix}.   
 \end{equation*}
 for all $v:= (u, \varrho)\in V$ by exactly the same arguments as in
 Lemmas 3.15 - 3.18 of \cite{Mue13}, see also
 \cite{AGMPreprint1}. Moreover, $-A'(0)$ is the operator defined by
 the right-hand side of \eqref{eq:LinearFlow1} with $\rho$ replaced by
 $u$ and without the integral-term as well as \eqref{eq:LinearFlow2}
 with $\rho$ replaced by $\varrho$.  The integral term arises as
 linearization of $F \in C^1(V,X_0)$ defined by $F(v)=
 (F_1(v),F_2(v))^T$ and
\begin{align*}
	F_1(u,\varrho) &:= -\frac{\bar{H}(u(t,q))}{n_\Gamma(u(t,q)) \cdot \d_w \Psi(q,u(t,q))}, \\
	F_2(u,\varrho) &:= 0
\end{align*}
for all $v=(u,\varrho)\in V$, i.e.,
\begin{equation*}
  F'(0)
  \begin{pmatrix}
    u,\varrho
  \end{pmatrix}
=
  \begin{pmatrix}
    - \mint_{\Gamma^*} (\Delta_{\Gamma^*} + |\sigma^*|^2) u \dH^2\\
    0
  \end{pmatrix}
\end{equation*}
for all $(u,\varrho)\in V$. Altogether $A_0:= A'(0)-F'(0)$ is the linearization of \eqref{eq:AbstractEvolutionEquation1} without the time derivative. It is the operator from (2.5) of \cite{PSZ09} adopted to our case $v^* \equiv 0$. Its spectral properties are crucial for the stability result below.
Finally, if we define  $v_0 := (\rho_0, \rho_0|_{\partial\Gamma^\ast})$,
(\ref{eq:AbstractEvolutionEquation1})-(\ref{eq:AbstractEvolutionEquation2}) is equivalent to \eqref{eq:Flow1}-\eqref{eq:Flow2} with $\Gamma_\rho(0)=\Gamma_{\rho_0}$.

We want to prove stability of SSCs, which means that we consider $v^* \equiv 0 \in \mathcal{E}$ parametrized over the SSC $\Gamma^*$, where $\mathcal{E}$ is the set of equilibria
\begin{align}\label{eq:Equilibria}
	\mathcal{E} := \left\{v \in V \cap X_1 \mid A(v) v = F(v)\right\} \subseteq V \cap X_1.
\end{align}
Clearly $\mathcal{E}$ is at least $2$-dimensional since we can shift any stationary surface in $x$- and $y$-direction without changing the curvatures, surface area and contact angle. That we consider $v^* \equiv 0$ also explains why our notation differs slightly from that of \cite{PSZ09}. In our special case there is no difference between what is called $v$ and $u$ in \cite{PSZ09}.

In Section \ref{sec:MaxReg} we will show in Theorem \ref{thm:Semigroup1} that $A'(0)$, which is $A_0$ without the non-local part $F'(0)$, has maximal $L_p$-regularity. This enables us to use Theorem 2.1 of \cite{PSZ09} which in our situation reads as follows.

\begin{thm}[GPLS]\label{thm:GPLStability}
Let $4 < p < \infty$ and suppose that $v^* \equiv 0$ is normally stable, i.e. \\
(a) near $v^*$ the set of equilibria $\mathcal{E}$ is a $C^1$-manifold in $X_1$, \\
(b) the tangent space of $\mathcal{E}$ at $v^*$ is given by $\mathcal{N}(A_0)$, \\
(c) $0$ is a semi-simple eigenvalue of $A_0$, i.e. $\mathcal{N}(A_0) \oplus \mathcal{R}(A_0) = X_0$, \\
(d) $\sigma(A_0) \setminus \{0\} \subseteq \C_+$. \\
Then $v^* \equiv 0$ is stable in $X_\gamma$ and there exists $\delta > 0$ such that the unique solution $v(t)$ of (\ref{eq:AbstractEvolutionEquation1})-(\ref{eq:AbstractEvolutionEquation2}) with initial value $v_0 \in X_\gamma$ satisfying $\norm{v_0}_{X_\gamma} < \delta$ exists on $\R^+$ and converges at an exponential rate in $X_\gamma$ to some $v_\infty \in \mathcal{E}$.
\end{thm}

\section{Maximal regularity}\label{sec:MaxReg}

In a first step we want to show that for fixed $T > 0$ the flow
\begin{align}\label{eq:LocalFlow1}
	\d_t \rho(t) & = \Delta_{\Gamma^*} \rho(t) + |\sigma^*|^2 \rho(t){+f(t)} \hspace*{52.5mm} \text{ in } [0,T] \times \Gamma^*, \\ \label{eq:LocalFlow2}
	\d_t \rho(t) & = -\sin(\alpha^*)^2 (n_{\d \Gamma^*} \cdot \nabla_{\Gamma^*} \rho(t)) + \sin(\alpha^*) \cos(\alpha^*) II_{\Gamma^*}(n_{\d \Gamma^*},n_{\d \Gamma^*}) \rho(t) \notag \\
					 & + b \sin(\alpha^*) \rho_{\sigma\sigma}(t) - b \sin(\alpha^*) \kappa_{\d D^*} \skp{\vec{\tau}^*}{(n_{\d D^*})_\sigma} \rho(t) {+g(t)} \qquad \text{ on } [0,T] \times \d \Gamma^*, \\ \label{eq:LocalFlow3}
	\rho(0) 		 & = \rho_0 \hspace*{78.5mm} \text{ in } \Gamma^*,
\end{align}
which is (\ref{eq:LinearFlow1})-(\ref{eq:LinearFlow2}) without the non-local part and an additional initial condition, has a unique solution. 

\begin{rem}\label{rem:PerturbationOperator}
In our first step we will not consider the non-local term of (\ref{eq:LinearFlow1}), which is given by the operator
\begin{align*}
	\mathcal{P}(\bullet) & := \mint_{\Gamma^*} (\Delta_{\Gamma^*} + |\sigma^*|^2) \bullet \dH^2.
\end{align*}
Later we will show that $\mathcal{P}$ is only a lower order perturbation of the original differential operator and does not effect the result.
\end{rem}


Now we want to move on to the more important considerations about the non-local part, which we ignored in (\ref{eq:LocalFlow1})-(\ref{eq:LocalFlow3}), but has to be included for the flow (\ref{eq:LinearFlow1})-(\ref{eq:LinearFlow2}). The basic ingredient will be a perturbation result of semigroup theory and the time-independence of the operators $\mathcal{A}$, $\mathcal{B}_0$, $\mathcal{B}_1$, $\mathcal{C}_0$ and $\mathcal{C}_1$.

We define a linear operator associated to \eqref{eq:LinearFlow1}-\eqref{eq:LinearFlow2} as
\begin{alignat*}{1}
	A&: \mathcal{D}(A) \longrightarrow W\,,\\
        A
        \begin{pmatrix}
          \rho \\
          \tilde{\rho}
        \end{pmatrix}& 
= 
	\begin{pmatrix}
		-\Delta_{\Gamma^*} \rho - |\sigma^*|^2 \rho \\[1ex]
		\sin(\alpha^*)^2 (n_{\d \Gamma^*} \cdot \nabla_{\Gamma^*} \tilde{\rho}) - \sin(\alpha^*) \cos(\alpha^*) II_{\Gamma^*}(n_{\d \Gamma^*},n_{\d \Gamma^*}) \tilde\rho \notag \\
					  - b \sin(\alpha^*) \tilde\rho_{\sigma\sigma}(t) + b \sin(\alpha^*) \kappa_{\d D^*} \skp{\vec{\tau}^*}{(n_{\d D^*})_\sigma} \tilde\rho(t)
	\end{pmatrix} 
\end{alignat*}
for all $(\rho,\tilde\rho)^T\in \mathcal{D}(A)$ with domain 
\begin{align*}
	\mathcal{D}(A) & := \left\{\left.(\rho,\tilde{\rho})^T \in W^2_p(\Gamma^*;\R) \times W^{3-\frac{1}{p}}_p(\d \Gamma^*;\R)
									 \right| \rho|_{\d \Gamma^*} = \tilde{\rho}\right\}
\end{align*}
equipped with the $W^2_p \times W^{3-\frac{1}{p}}_p$-norm and the codomain is
\begin{align*}
	W & = L_p(\Gamma^*;\R) \times W^{1-\frac{1}{p}}_p(\d \Gamma^*;\R).
\end{align*}
Hence $A=A'(0)$.
{\begin{rem}\label{rem:EquivalenceOfNorms}
    We note that the norm on $\mathcal{D}(A)$ is equivalent to the graph norm, which can be seen as follows: By Theorem~\ref{thm:Semigroup1} below $-A$ generates an analytic semigroup. Therefore there is some $\lambda>0$ such that $\lambda +A\colon \mathcal{D}(A)\to W$ is invertible. This implies that there is some $C>0$ such that $\|u\|_{W^2_p\times W^{3-\frac1p}_p}\leq C\left(\|u\|_{W}+\|Au\|_{W}\right)$ for all $u\in \mathcal{D}(A)$. Hence the graph norm is stronger than the $W^2_p\times W^{3-\frac1p}_p$-norm on $\mathcal{D}(A)$. By the open mapping theorem both norms are equivalent.
  \end{rem}}


For this operator $A$ we get the following statement from \cite{DPZ08}.

\begin{thm}\label{thm:Semigroup1}
Let $3 < p < \infty$. Then the operator $-A$ generates an analytic semigroup in $W$, which has the property of maximal $L_p$-regularity on each finite interval $J = [0,T]$. Moreover, there is some $\omega \geq 0$ such that $-(A + \omega \id)$ has maximal $L_p$-regularity on the half-line $\R_+$.
\end{thm}
\begin{proof}
{This result follows from Theorem 2.2 of \cite{DPZ08} applied to the given situation. We refer to \cite{AGMPreprint1} for more details on the application of this result.}
\end{proof}

Now we use a perturbation argument for generators of analytic semigroups taken from \cite{Paz83} to treat the non-local part $\mathcal{P}$. This is the essential ingredient needed to proof the existence of solutions for the flow (\ref{eq:LinearFlow1})-(\ref{eq:LinearFlow2}).

\begin{lemma}\label{lem:SemigroupPerturbation}
Let {$-A$} be the generator of an analytic semigroup on $X$. Let $P$ be a closed linear operator satisfying $\mathcal{D}(P) \supseteq \mathcal{D}(A)$ and
\begin{align}\label{eq:PerturbationProp}
	\norm{Px}_X \leq \epsilon \norm{Ax}_X + M \norm{x}_X \qquad \forall \ x \in \mathcal{D}(A).
\end{align}
Then there is some $\epsilon_0 > 0$ such that, if $0 \leq \epsilon \leq \epsilon_0$, then {$-A + P$} is the generator of an analytic semigroup.
\end{lemma}
\begin{proof}
Can be found in \cite{Paz83} on page 80.
\end{proof}

In our case the perturbation operator $P$ reads as follows
\begin{align*}
	P: W^2_p(\Gamma^*;\R) \times W^{3-\frac{1}{p}}_p(\d \Gamma^*;\R) \longrightarrow \R \times \{0\}: \begin{pmatrix}
																																		  \rho \\
																																		  \tilde{\rho}
																																	  \end{pmatrix}
	\longmapsto \begin{pmatrix}
						P_1 & \mathbb{O} \\
						\mathbb{O} & \mathbb{O}
					\end{pmatrix} \begin{pmatrix}
										  \rho \\
										  \tilde{\rho}
									  \end{pmatrix},
\end{align*}
where the operator $P_1$ is defined as
\begin{align*}
	P_1(\rho) := -\mint_{\Gamma^*} (\Delta_{\Gamma^*} + |\sigma^*|^2) \rho \dH^2.
\end{align*}
Due to the fact that $\Gamma^*$ is bounded we can embed the space $\R$ into $L_p(\Gamma^*;\R)$. Therefore, we can consider $P$ as an operator
\begin{align*}
	P: \mathcal{D}(P) \longrightarrow L_p(\Gamma^*;\R) \times W^{1-\frac{1}{p}}_p(\d \Gamma^*;\R)
\end{align*}
with $\mathcal{D}(P) := W^2_p(\Gamma^*;\R) \times W^{3-\frac{1}{p}}_p(\d \Gamma^*;\R) \supseteq \mathcal{D}(A)$ as required in Lemma \ref{lem:SemigroupPerturbation}. The argument $\R \hookrightarrow L_p(\Gamma^*;\R)$ also shows that $P$ is a closed linear operator. Now our goal is to prove that equation (\ref{eq:PerturbationProp}) is valid with arbitrarily small $\epsilon$. Hence, we would see $-A + P$ is also a generator of an analytic semigroup. The necessary steps to achieve this aim will be distributed to several lemmas. For a more convenient notation we define the spaces $V$ and $W$ to be
\begin{align*}
	V & := W^2_p(\Gamma^*;\R) \times W^{3-\frac{1}{p}}_p(\d \Gamma^*;\R).
\end{align*}

\begin{lemma}\label{lem:NormPEstimate1}
For all $x \in \mathcal{D}(A)$ one has the estimate
\begin{align}\label{eq:NormPEstimate1}
	\norm{Px}_W \leq c \norm{x}_V^\vartheta \norm{x}_W^{1-\vartheta}
\end{align}
for some $\vartheta \in (0,1)$.
\end{lemma}
\begin{proof}
First we see
\begin{align*}
	\norm{P_1\rho}_{L_p(\Gamma^*)} = \left(\int_{\Gamma^*} \left|P_1\rho\right|^p \dH^2\right)^{\frac{1}{p}} = A(\Gamma^*)^{\frac{1}{p} - 1} \left|\int_{\Gamma^*} \Delta_{\Gamma^*} \rho + |\sigma^*|^2 \rho \dH^2\right|.
\end{align*}
Due to the compactness of $\Gamma^* \cup \d \Gamma^*$ and the smoothness of $\Gamma^*$ up to the boundary we have $|\sigma^*|^2 \leq c$. Hence we continue with the estimate from above
\begin{align*}
	\norm{P_1\rho}_{L_p(\Gamma^*)}
		& \leq A(\Gamma^*)^{\frac{1}{p} - 1} \left(\left|\int_{\Gamma^*} \Delta_{\Gamma^*} \rho \dH^2\right| + c \int_{\Gamma^*} |\rho| \dH^2\right) \\
		& = A(\Gamma^*)^{\frac{1}{p} - 1} \left(\left|\int_{\d \Gamma^*} n_{\d \Gamma^*} \cdot \nabla_{\Gamma^*} \rho \dH^1\right| + c \norm{\rho}_{L_1(\Gamma^*)}\right) \\
		& \leq A(\Gamma^*)^{\frac{1}{p} - 1} \left(\hat{c} \norm{\nabla_{\Gamma^*} \rho}_{L_1(\d \Gamma^*)} + c \norm{\rho}_{L_1(\Gamma^*)}\right) \\
		& \leq \tilde{c} \left(\norm{\nabla_{\Gamma^*} \rho}_{L_p(\d \Gamma^*)} + \norm{\rho}_{L_p(\Gamma^*)}\right),
\end{align*}
where we used Gauss' theorem on manifolds in the second line. For every finite measure space $(\Omega,\mu)$ and every $\epsilon > 0$ one has $W^\epsilon_p(\Omega) \hookrightarrow L_p(\Omega)$ and thus we obtain
\begin{align*}
	\norm{P_1\rho}_{L_p(\Gamma^*)} \leq \tilde{c} \left(\norm{\nabla_{\Gamma^*} \rho}_{L_p(\d \Gamma^*)} + \norm{\rho}_{L_p(\Gamma^*)}\right) \leq \hat{c} \left(\norm{\nabla_{\Gamma^*} \rho}_{W^\epsilon_p(\d \Gamma^*)} + \norm{\rho}_{W^{1+\frac{1}{p}+\epsilon}_p(\Gamma^*)}\right).
\end{align*}
Furthermore, the trace operator $\gamma_0$ is linear and bounded from $W^s_p(\Omega)$ to $W^{s - \frac{1}{p}}_p(\d \Omega)$ for every $s > \frac{1}{p}$ and we have
\begin{align}\label{eq:P1Estimate}
	\norm{P_1\rho}_{L_p(\Gamma^*)}
		& \leq \hat{c} \left(\norm{\nabla_{\Gamma^*} \rho}_{W^\epsilon_p(\d \Gamma^*)} + \norm{\rho}_{W^{1+\frac{1}{p}+\epsilon}_p(\Gamma^*)}\right) \notag \\
		& \leq c' \left(\norm{\nabla_{\Gamma^*}
                    \rho}_{W^{\frac{1}{p}+\epsilon}_p(\Gamma^*)} +
                  \norm{\rho}_{W^{1+\frac{1}{p}+\epsilon}_p(\Gamma^*)}\right)
                \leq c
                \norm{\rho}_{W^{1+\frac{1}{p}+\epsilon}_p(\Gamma^*)} \,.
\end{align}
Using that $\Gamma^\ast$ is diffeomorphic to a bounded smooth domain
the Example 2.12 from \cite{Lun09} shows that $W^{1 + \frac{1}{p} +
  \epsilon}_p(\Gamma^*)$ is an interpolation space of exponent
$\vartheta = \frac{1}{2} (1 + \frac{1}{p} + \epsilon) \in (0,1)$ with
respect to $\left(L_p(\Gamma^*),W^2_p(\Gamma^*)\right)$, where we
assume w.l.o.g. $\epsilon < 1 - \frac{1}{p}$. This leads to
\begin{align*}
	\norm{Px}_W & \leq \tilde{c} \norm{\rho}_{W^{1+\frac{1}{p}+\epsilon}_p(\Gamma^*)} \leq c \norm{\rho}_{W^2_p(\Gamma^*)}^\vartheta \norm{\rho}_{L_p(\Gamma^*)}^{1-\vartheta} \\
					& \leq c \left(\norm{\rho}_{W^2_p(\Gamma^*)} + \norm{\tilde{\rho}}_{W^{3-\frac{1}{p}}_p(\d \Gamma^*)}\right)^\vartheta \left(\norm{\rho}_{L_p(\Gamma^*)} + \norm{\tilde{\rho}}_{W^{1-\frac{1}{p}}_p(\d \Gamma^*)}\right)^{1-\vartheta} = c \norm{x}_V^\vartheta \norm{x}_W^{1-\vartheta}
\end{align*}
and shows the desired result.
\end{proof}

\begin{lemma}\label{lem:NormPEstimate2}
For all $x \in \mathcal{D}(A)$ one has the estimate
\begin{align}\label{eq:NormPEstimate2}
	\norm{Px}_W \leq c \left(\norm{Ax}_W^\vartheta \norm{x}_W^{1-\vartheta} + \norm{x}_W\right)
\end{align}
for some $\vartheta \in (0,1)$.
\end{lemma}
\begin{proof}
{First of all, because of Remark~\ref{rem:EquivalenceOfNorms} we have
\begin{align*}
 	\norm{x}_V \leq  c \left(\norm{x}_W + \norm{Ax}_W\right).
 \end{align*}}
Using Lemma \ref{lem:NormPEstimate1}, we finally arrive at
\begin{align*}
	\norm{Px}_W & \leq c \norm{x}_V^\vartheta \norm{x}_W^{1-\vartheta} \leq c \tilde{c}^\vartheta \left(\norm{x}_W + \norm{Ax}_W\right)^\vartheta \norm{x}_W^{1-\vartheta} \\
					& \leq \hat{c} \left(\norm{x}_W^\vartheta + \norm{Ax}_W^\vartheta\right) \norm{x}_W^{1-\vartheta} = \hat{c} \norm{x}_W + \hat{c} \norm{Ax}_W^\vartheta \norm{x}_W^{1-\vartheta},
\end{align*}
where we used $(a + b)^\vartheta \leq (a^\vartheta + b^\vartheta)$.
\end{proof}

\begin{thm}\label{thm:Semigroup2}
Let $3 < p < \infty$. Then the operator $-A + P$ generates an analytic semigroup in $W$.
\end{thm}
\begin{proof}
We will use Lemma \ref{lem:SemigroupPerturbation}. Because of Theorem \ref{thm:Semigroup1}, we know that $-A$ generates an analytic semigroup. As stated immediately after the definition of $P$, the assumptions ``$\mathcal{D}(P) \supseteq \mathcal{D}(A)$'' and ``$P$ closed'' are satisfied and therefore only (\ref{eq:PerturbationProp}) remains to be proven. For $\vartheta \in (0,1)$ as in Lemma \ref{lem:NormPEstimate2} we define $p' := \frac{1}{\vartheta}$ and $q' := \frac{1}{1 - \vartheta}$, which gives $1 < p',q' < \infty$ and $\frac{1}{p'} + \frac{1}{q'} = \vartheta + 1 - \vartheta = 1$. Young's inequality with $\varepsilon$ leads to
\begin{align*}
	\norm{Px}_W & \leq c \norm{Ax}_W^\vartheta \norm{x}_W^{1-\vartheta} + c \norm{x}_W \leq c \epsilon \left(\norm{Ax}_W^\vartheta\right)^{\frac{1}{\vartheta}} 
					  + c \left(\frac{\vartheta}{\epsilon}\right)^{q-1} (1 - \vartheta) \left(\norm{x}_W^{1-\vartheta}\right)^{\frac{1}{1-\vartheta}} + c \norm{x}_W \\
					& = c \epsilon \norm{Ax}_W + M(\vartheta,\epsilon) \norm{x}_W
\end{align*}
in which we used Lemma \ref{lem:NormPEstimate2} in the first inequality. Since $\epsilon > 0$ can be chosen arbitrarily small, we get the desired statement (\ref{eq:PerturbationProp}) of Lemma \ref{lem:SemigroupPerturbation}.
\end{proof}

\section{Application}\label{sec:Application}

In the process of using the GPLS, it will be necessary to make use of a better suited parametrization of the SSC $\Gamma^*$. We will assume w.l.o.g. that the center of the SSC $\Gamma^*$ lies on the $z$-axis and has height $H^* \in (-R^*,R^*)$ over or under the $x$-$y$-plane. The perfect fit for SCs are spherical coordinates shifted in $z$-direction by $H^*$, which will be introduced now.

Let $a$ and $b$ be given as in (\ref{eq:ParameterEquations}). Then we know by the considerations in Section \ref{sec:SphericalCaps} that for arbitrary $\alpha^* \in I_\alpha$ there is some $r^* \in I_r$ such that
\begin{align*}
	\cos(\alpha^*) = \frac{b}{r^*} - a \qquad \text{ and } \qquad \sin(\alpha^*) = \sqrt{1 - \left(\frac{b}{r^*} - a\right)^2}
\end{align*}
as well as $R^* \in (0, \infty)$ and $H^* \in (-R^*,R^*)$ to satisfy
\begin{align*}
	R^* := \frac{r^*}{\sin(\alpha^*)} \qquad \text{ and } \qquad H^* := R^* \cos(\alpha^*).
\end{align*}
Then the parametrization of $\Gamma^*$ reads as
\begin{align}\label{eq:SphericalCoordinates}
	P(\phi,\theta) := \begin{pmatrix}
								R^* \sin(\phi) \sin(\theta) \\
								R^* \cos(\phi) \sin(\theta) \\
								R^* \cos(\theta) + H^*
							\end{pmatrix} \quad \text{ with } \phi \in [0, 2\pi] \text{ and } \theta \in [0, \pi - \alpha^*].
\end{align}
We can use this to calculate the following quantities in the case of $\Gamma^*$ being an SSC as follows
\begin{align*}
	n_{\Gamma^*} &= \begin{pmatrix}
							 \sin(\phi) \sin(\theta) \\
							 \cos(\phi) \sin(\theta) \\
							 \cos(\theta)
						 \end{pmatrix}, &	\sqrt{g} &= {R^*}^2 \sin(\theta), \\
	\kappa_1^* &= \kappa_2^* = -\frac{1}{R^*}, & |\sigma^*|^2 &= \frac{2}{{R^*}^2}, \\
	H_{\Gamma^*} &= -\frac{2}{R^*}, & K_{\Gamma^*} &= \frac{1}{{R^*}^2}, \\
	n_{\d \Gamma^*} &= \begin{pmatrix}
								 -\sin(\phi) \cos(\alpha^*) \\
								 -\cos(\phi) \cos(\alpha^*) \\
								 -\sin(\alpha^*)
							 \end{pmatrix}, &	\vec{\tau}^* &= \begin{pmatrix}
																			 \cos(\phi)  \\
																			 -\sin(\phi) \\
																			 0
																		 \end{pmatrix}, \\
	n_{\d D^*} &= \begin{pmatrix}
						  \sin(\phi) \\
						  \cos(\phi) \\
						  0
					  \end{pmatrix}, & \vec{\kappa}^* &= \frac{1}{R^* \sin(\alpha^*)} \begin{pmatrix}
																												-\sin(\phi) \\
																												-\cos(\phi) \\
																												0
																											\end{pmatrix}, \\
	\kappa_{\d D^*} &= \skp{\vec{\kappa}^*}{n_{\d D^*}} = -\frac{1}{R^* \sin(\alpha^*)}, \\
	\nabla_{\Gamma^*} &= \frac{1}{R^* \sin(\theta)} \begin{pmatrix}
																		\cos(\phi) \\
																		-\sin(\phi) \\
																		0
																	\end{pmatrix} \d_\phi + \frac{1}{R^*} \begin{pmatrix}
																														  \sin(\phi) \cos(\theta) \\
																														  \cos(\phi) \cos(\theta) \\
																														  -\sin(\theta)
																													  \end{pmatrix} \d_\theta, & & \\
	\Delta_{\Gamma^*} &= \frac{1}{{R^*}^2 \sin(\theta)^2} \d^2_{\phi \phi} + \frac{1}{{R^*}^2} \d^2_{\theta \theta} + \frac{1}{{R^*}^2} \cot(\theta) \d_\theta. & &
\end{align*}

Before we can check the assumptions of the GPLS it will be necessary to determine the nulls pace of the operator $A_0$. For more details on the calculations in the upcoming considerations, we refer to \cite{Mue13}. The first step is to fit the equations (\ref{eq:LinearFlow1})-(\ref{eq:LinearFlow2}) to the situation of $\Gamma^*$ being an SSC with the above parametrization. Here we see that the first component of $-A_0 \rho$ has the form
\begin{align}\label{eq:InteriorForSSC}
	-(A_0 \rho)^{(1)} = \Delta_{\Gamma^*} \rho + |\sigma^*|^2 \rho - \mint_{\Gamma^*} (\Delta_{\Gamma^*} + |\sigma^*|^2) \rho \dH^2.
\end{align}
Searching for solutions of $0 = -(A_0 \rho)^{(1)}$ we immediately see that $\Delta_{\Gamma^*} \rho + |\sigma^*|^2 \rho = const.$ has to hold. And vice versa, if $\Delta_{\Gamma^*} \rho + |\sigma^*|^2 \rho$ is constant we get $-(A_0 \rho)^{(1)} = 0$. Therefore it is equivalent to solve $c = \Delta_{\Gamma^*} \rho + |\sigma^*|^2 \rho$ instead of $0 = -(A_0 \rho)^{(1)}$. Transforming the equation with respect to the parametrization from above we have to solve
\begin{align}\label{eq:NullspaceDGL}
	c = \frac{1}{\sin(\theta)^2} \rho_{\phi \phi} + \rho_{\theta \theta} + \cot(\theta) \rho_{\theta} + 2 \rho \qquad \text{ in } (0, 2\pi) \times (0, \pi - \alpha^*),
\end{align}
where the missing ${R^*}^2$ is included in the constant on the left side. For the boundary component we get
\begin{align}\label{eq:BCForSSC}
	-(A_0 \rho)^{(2)} &= -\sin(\alpha^*)^2 (n_{\d \Gamma^*} \cdot \nabla_{\Gamma^*} \rho)
   						 + \sin(\alpha^*) \cos(\alpha^*) II_{\Gamma^*}(n_{\d \Gamma^*},n_{\d \Gamma^*}) \rho \notag \\
							&+ b \sin(\alpha^*) \rho_{\sigma\sigma}
							 - b \sin(\alpha^*) \kappa_{\d D^*} \skp{\vec{\tau}^*}{(n_{\d D^*})_\sigma} \rho.
\end{align}
Using the calculations above we have
\begin{align}
	n_{\d \Gamma^*} \cdot \nabla_{\Gamma^*} \rho & = \left.\frac{1}{R^*} \rho_\theta\right|_{\theta = \pi - \alpha^*}, \notag \\ \label{eq:SSCQuantities1}
	II_{\Gamma^*}(n_{\d \Gamma^*},n_{\d \Gamma^*}) & = II_{\Gamma^*}\left(\left.\frac{P_\theta}{\norm{P_\theta}}\right|_{\theta = \pi - \alpha^*},\left.\frac{P_\theta}{\norm{P_\theta}}\right|_{\theta = \pi - \alpha^*}\right) = -\frac{1}{R^*}, \\
	\rho_{\sigma \sigma} & = \left.\frac{1}{\norm{P_\phi}} \d_\phi \left(\frac{\rho_\phi}{\norm{P_\phi}}\right)\right|_{\theta = \pi - \alpha^*} = \frac{1}{{R^*}^2 \sin(\alpha^*)^2} \rho_{\phi \phi}, \notag \\ \label{eq:SSCQuantities2}
	\skp{\vec{\tau}^*}{(n_{\d D^*})_\sigma} & = \left.\skp{\vec{\tau}^*}{\frac{(n_{\d D^*})_\phi}{\norm{P_\phi}}}\right|_{\theta = \pi - \alpha^*} = \frac{1}{R^* \sin(\alpha^*)}
\end{align}
and plugging this into the equation for $-(A_0 \rho)^{(2)}$ we end up with
\begin{align*}
	-(A_0 \rho)^{(2)} = \frac{\sin(\alpha^*)}{R^*} \left(-\sin(\alpha^*) \rho_\theta - \cos(\alpha^*) \rho + \frac{b}{R^* \sin(\alpha^*)^2} \rho_{\phi \phi} + \frac{b}{R^* \sin(\alpha^*)^2} \rho\right).
\end{align*}
We divide by $\frac{\sin(\alpha^*)}{R^*} \neq 0$ and obtain the first boundary condition for the nulls pace to be
\begin{align}\label{eq:NullspaceBC1}
	0 = \left.\frac{b}{R^* \sin(\alpha^*)^2} \rho_{\phi \phi} + \frac{b}{R^* \sin(\alpha^*)^2} \rho - \sin(\alpha^*) \rho_\theta - \cos(\alpha^*) \rho\right|_{\theta = \pi - \alpha^*} \quad \text{ on } [0, 2\pi].
\end{align}
Because we transformed $A_0$ into spherical coordinates $(\phi,\theta) \in [0, 2\pi] \times [0, \pi - \alpha^*]$, we still have to impose three more boundary conditions. These represent the compatibility conditions on the ``new'' boundaries $\phi = 0$, $\phi = 2\pi$ and $\theta = 0$ that have not been present as we parametrized over $\Gamma^*$.
The second and third boundary condition represent the periodicity in $\phi$ namely
\begin{align}\label{eq:NullspaceBC2}
	0 &= \rho|_{\phi = 0} - \rho|_{\phi = 2\pi} & &\quad \text{ on } [0, \pi - \alpha^*], \\ \label{eq:NullspaceBC3}
	0 &= \rho_\phi|_{\phi = 0} - \rho_\phi|_{\phi = 2\pi} & &\quad \text{ on } [0, \pi - \alpha^*].
\end{align}
The fourth boundary condition shall guarantee continuity in the ``north pole'' of the SSC. Here we demand 
\begin{align}\label{eq:NullspaceBC4}
	const. = \rho|_{\theta = 0} \qquad \text{ on } [0, 2\pi].
\end{align}

Combining the equations (\ref{eq:NullspaceDGL}) and (\ref{eq:NullspaceBC1})-(\ref{eq:NullspaceBC4}) we have to solve the system
\begin{align}\label{eq:NullspaceSystem1}
	c &= \frac{1}{\sin(\theta)^2} \rho_{\phi \phi} + \rho_{\theta \theta} + \cot(\theta) \rho_{\theta} + 2 \rho & &\quad \text{in } (0, 2\pi) \times (0, \pi - \alpha^*), \\ \label{eq:NullspaceSystem2}
	0 &= \frac{b}{R^* \sin(\alpha^*)^2} (\rho_{\phi \phi} + \rho) - \sin(\alpha^*) \rho_\theta - \cos(\alpha^*) \rho & &\quad \text{on } [0, 2\pi] \times \{\pi - \alpha^*\}, \\ \label{eq:NullspaceSystem3}
	0 &= \rho|_{\phi = 0} - \rho|_{\phi = 2\pi} & &\quad \text{on } [0, \pi - \alpha^*], \\ \label{eq:NullspaceSystem4}
	0 &= \rho_\phi|_{\phi = 0} - \rho_\phi|_{\phi = 2\pi} & &\quad \text{on } [0, \pi - \alpha^*], \\ \label{eq:NullspaceSystem5}
	const. &= \rho|_{\theta = 0} & &\quad \text{on } [0, 2\pi]
\end{align}
to get all elements in the nulls pace of $A_0$.

First we find a special solution of the inhomogeneous system by an educated guess. It is an easy calculation to verify that $\rho^s$ given by
\begin{align}\label{eq:SpecialRho}
	\rho^s(\phi,\theta) := 1 + c_\alpha \cos(\theta)
\end{align}
with
\begin{align*}
	 c_\alpha := \frac{R^* \cos(\alpha^*) \sin(\alpha^*)^2 - b}{R^* \sin(\alpha^*)^2 - b \cos(\alpha^*)}
\end{align*}
satisfies the conditions (\ref{eq:NullspaceSystem1})-(\ref{eq:NullspaceSystem5}). Obviously, this is only possible if $R^* \sin(\alpha^*)^2 \neq b \cos(\alpha^*)$. We claim that for $R^* \sin(\alpha^*)^2 = b \cos(\alpha^*)$ there exists no function that satisfies (\ref{eq:NullspaceSystem1})-(\ref{eq:NullspaceSystem5}) with a $c \neq 0$ and will prove that fact later on in Lemma \ref{lem:CriticalCase}.

A separation ansatz $\rho(\phi,\theta) = f(\phi) g(\theta)$ is common practice to solve such a homogeneous system (\ref{eq:NullspaceSystem1})-(\ref{eq:NullspaceSystem5}). But before we start with that, we want to justify this separation of variables following the ideas from Lecture 4 and 11 of \cite{Sai07}.

The operator $\Delta^B: X_1 \longrightarrow X_0$ is defined as
\begin{align*}
	\Delta^B \rho := \begin{pmatrix}
							  -\Delta_{\Gamma^*} \rho^{(1)} - |\sigma^*|^2 \rho^{(1)} \\
							  \sin(\alpha^*)^2 (n_{\d \Gamma^*} \cdot \nabla_{\Gamma^*} \rho^{(1)})
   						  + \frac{\sin(\alpha^*) \cos(\alpha^*)}{R^*} \rho^{(2)} - b \sin(\alpha^*) \rho^{(2)}_{\sigma\sigma}
							  - \frac{b}{{R^*}^2 \sin(\alpha^*)} \rho^{(2)}
						  \end{pmatrix}
\end{align*}
and is symmetric with respect to the inner product defined by
\begin{align*}
	\skp{u}{v}_{\tilde{L_2}} := \int_{\Gamma^*} u^{(1)} v^{(1)} \dH^2 + \int_{\d \Gamma^*} \frac{1}{\sin(\alpha^*)^2} u^{(2)} v^{(2)} \dH^1
\end{align*}
as one can see from straightforward calculations. Therefore all eigenvalues are real and the eigenfunctions corresponding to different eigenvalues are orthogonal with respect to this inner product.

\begin{rem}\label{rem:ImportantInnerProduct}
This $\tilde{L_2}$-inner product will also play an important role later on, while proving the solvability of (\ref{eq:EllipticPDE1})-(\ref{eq:EllipticPDE2}).
\end{rem}

In $(\phi,\theta)$-coordinates $\Delta^B$ is given as
\begin{align*}
	\Delta^B \rho = \begin{pmatrix}
							 -\dfrac{1}{{R^*}^2 \sin(\theta)^2} \rho^{(1)}_{\phi \phi} - \dfrac{1}{{R^*}^2} \rho^{(1)}_{\theta \theta} - \dfrac{1}{{R^*}^2} \cot(\theta) \rho^{(1)}_\theta - \dfrac{2}{{R^*}^2} \rho^{(1)} \\
							 \dfrac{\sin(\alpha^*)^2}{R^*} \rho^{(1)}_\theta + \dfrac{\sin(\alpha^*) \cos(\alpha^*)}{R^*} \rho^{(2)} - \dfrac{b}{{R^*}^2 \sin(\alpha^*)} (\rho^{(2)}_{\phi \phi} + \rho^{(2)})
						 \end{pmatrix},
\end{align*}
where we have to impose the boundary conditions $\rho|_{\phi = 0} = \rho|_{\phi = 2\pi}$ and $\rho_\phi|_{\phi = 0} = \rho_\phi|_{\phi = 2\pi}$. We will decompose $\Delta^B$ into a part corresponding to differentiation with respect to $\phi$ and another part corresponding to differentiation with respect to $\theta$. For $f: [0,2\pi] \longrightarrow \R^2: \phi \longrightarrow (f^{(1)}(\phi), f^{(2)}(\phi))$ the $\phi$-part shall be given as
\begin{align*}
	\Delta_\phi f := \begin{pmatrix}
							  -f^{(1)}_{\phi\phi} \\
							  -f^{(2)}_{\phi\phi}
						  \end{pmatrix}
\end{align*}
with its boundary conditions $f(0) = f(2\pi)$ and $f_\phi(0) = f_\phi(2\pi)$. It is easy to see that the eigenvalues of this operator are $k^2$ for $k \in \N$. We use these eigenvalues of $\Delta_\phi$ to define the $\theta$-part of $\Delta^B$ as
\begin{align}\label{eq:ThetaPart}
	\Delta^k_\theta g := \begin{pmatrix}
									-\dfrac{1}{{R^*}^2} g^{(1)}_{\theta \theta} - \dfrac{1}{{R^*}^2} \cot(\theta) g^{(1)}_\theta - \dfrac{1}{{R^*}^2} \left(2 - \dfrac{k^2}{\sin(\theta)^2}\right) g^{(1)} \\
									\dfrac{\sin(\alpha^*)^2}{R^*} g^{(1)}_\theta(\pi - \alpha^*) + \dfrac{\sin(\alpha^*) \cos(\alpha^*)}{R^*} g^{(2)} - \dfrac{b (1 - k^2)}{{R^*}^2 \sin(\alpha^*)} g^{(2)}
								\end{pmatrix},
\end{align}
where $g^{(2)}$ is in $\R$ and $g^{(1)}$ is a function $g^{(1)}: [0, \pi - \alpha^*] \longrightarrow \R$ with $g^{(1)}(\pi - \alpha^*) = g^{(2)}$. Assume that we have an eigenpair $(k^2, f_k)$ of $\Delta_\phi$ and for this $k \in \N$ an eigenpair $(\mu_k, g_k)$ of $\Delta^k_\theta$. Then $(\mu_k, f_k g_k)$ is an eigenpair of $\Delta^B$, since
\begin{align}\label{eq:EigenpairsLaplaceB}
	\Delta^B (f_k g_k) = \mu_k f_k g_k
\end{align}
as one can easily check by straightforward calculations.

The next step in our separation ansatz justification is to show that there is an orthogonal basis of eigenfunctions of $\Delta^k_\theta$ in a certain space. We define a weighted $L_2$- and $W^1_2$-space via
\begin{align*}
	\skp{u}{v}_{\hat{L_2}} & := {R^*}^2 \int^{\pi - \alpha^*}_0 u^{(1)} v^{(1)} \sin(\theta) d\theta + \frac{R^*}{\sin(\alpha^*)} u^{(2)} v^{(2)}, \\
	\hat{L_2} & := \left\{f: [0, \pi - \alpha^*] \longrightarrow \R^2 \text{ measurable} \ \left| \ f^{(2)} \text{ is constant, } \norm{f}_{\hat{L_2}} := \sqrt{\skp{f}{f}_{\hat{L_2}}} < \infty\right.\right\}, \\
	\skp{u}{v}_{\hat{W^1_2}} & := \int^{\pi - \alpha^*}_0 u^{(1)}_\theta v^{(1)}_\theta \sin(\theta) d\theta + \int^{\pi - \alpha^*}_0 \frac{k^2}{\sin(\theta)} u^{(1)} v^{(1)} d\theta + \skp{u}{v}_{\hat{L_2}}, \\
	\hat{W^1_2} & := \left\{f \in \hat{L_2} \left| \d_\theta f \in \hat{L_2}, \norm{f}_{\hat{W^1_2}} := \sqrt{\skp{f}{f}_{\hat{W^1_2}}} < \infty, f^{(1)}(\pi - \alpha^*) = f^{(2)}\right.\right\}
\end{align*}
and a bilinear form $B: \hat{W^1_2} \times \hat{W^1_2} \longrightarrow \R$ by
\begin{align*}
	B(u,v) & := \int^{\pi - \alpha^*}_0 u^{(1)}_\theta v^{(1)}_\theta \sin(\theta) - \left(2 - \frac{k^2}{\sin(\theta)^2}\right) u^{(1)} v^{(1)} \sin(\theta) d\theta \\
			 & + \left(\cos(\alpha^*) - \frac{b (1 - k^2)}{R^* \sin(\alpha^*)^2}\right) u^{(2)} v^{(2)}.
\end{align*}
Then we obtain
\begin{align}\label{eq:BilinearDTheta}
	\skp{\Delta^k_\theta g}{h}_{\hat{L_2}} = B(g,h)
\end{align}
for all $g, h \in \hat{W^1_2}$. This bilinear form is bounded with respect to the norm defined on $\hat{W^1_2}$. Moreover, the modified bilinear form
\begin{align*}
	B_c: \hat{W^1_2} \times \hat{W^1_2} \longrightarrow \R: (u,v) \longmapsto B(u,v) + c\skp{u}{v}_{\hat{L_2}}
\end{align*}
is also bounded and in addition positive definite for
\begin{align*}
	c > \max \left\{\frac{2}{{R^*}^2}, \frac{b (1 - k^2)}{{R^*}^2 \sin(\alpha^*)} - \frac{\cos(\alpha^*) \sin(\alpha^*)}{R^*}\right\} > 0.
\end{align*}
Therefore $B_c$ satisfies all assumptions for the lemma of Lax-Milgram and there exists a bounded operator
\begin{align*}
	\left(\Delta^k_\theta + c \id\right)^{-1}: \hat{L_2} \longrightarrow \hat{W^1_2}
\end{align*}
corresponding to a weak solution operator for $(\Delta^k_\theta + c \id) g = f$ with $f \in \hat{L_2}$. We will show in Lemma \ref{lem:CompactEmbedding} that regardless of our modified definition of the $\hat{L_2}$- and $\hat{W^1_2}$-space the compact embedding $\hat{W^1_2} \hookrightarrow \hat{L_2}$ holds true as usual. Therefore
\begin{align*}
	\left(\Delta^k_\theta + c \id\right)^{-1}: \hat{L_2} \longrightarrow \hat{W^1_2} \hookrightarrow \hat{L_2}
\end{align*}
is a compact operator. By the spectral theorem for compact operators we know that $\left(\Delta^k_\theta + c \id\right)^{-1}$ has countably many eigenfunctions $(g^k_m)_{m \in \N}$, that form an orthonormal basis of $\hat{L_2}$. The eigenfunctions are invariant under inversion and shifting, hence also the eigenfunctions of $\Delta^k_\theta$ are an orthonormal basis of $\hat{L_2}$ as well.

\begin{rem}\label{rem:IsolatedEVs}
The spectral theorem for compact operators also states that the eigenvalues of $\left(\Delta^k_\theta + c \id\right)^{-1}$ form a sequence converging to zero. In particular, the eigenvalues have no accumulation point other than $0$. Therefore the eigenvalues of the non-inverted operator have no accumulation point. The shift of the eigenvalues by $c$ does not change this fact. Thus all eigenvalues of $\Delta^k_\theta$ and with them also the eigenvalues of $\Delta^B$ are isolated.
\end{rem}

It is well-known that also the eigenfunctions $(f_k)_{k \in \N}$ of $\Delta_\phi$, given by $\sin(k \phi)$ and $\cos(k \phi)$, form an orthogonal basis in $L_2([0, 2\pi])$.

Now we assume that there is an eigenfunction $u$ of $\Delta^B$ corresponding to the eigenvalue $\lambda$ that is not in the span of all functions that are in product form. Since we know that all eigenfunctions corresponding to different eigenvalues of $\Delta^B$ are orthogonal with respect to the $\tilde{L_2}$-inner product and $f_k g^k_m$ is an eigenfunction of $\Delta^B$, we see that for arbitrary $k, m \in \N$ we would obtain
\begin{align*}
	0 & = \skp{u}{f_k g^k_m}_{\tilde{L_2}} = \int^{\pi - \alpha^*}_0 \int^{2 \pi}_0 u^{(1)}(\phi,\theta) f_k^{(1)}(\phi) {g^k_m}^{(1)}(\theta) {R^*}^2 \sin(\theta) d\theta d\phi \\
	  & + \frac{1}{\sin(\alpha^*)^2} \int^{2 \pi}_0 u^{(2)}(\phi) f_k^{(2)}(\phi) {g^k_m}^{(2)} R^* \sin(\alpha^*) d\phi \\
	  & = {R^*}^2 \int^{\pi - \alpha^*}_0 \left(\int^{2 \pi}_0 u^{(1)}(\phi,\theta) f_k^{(1)}(\phi) d\phi\right) {g^k_m}^{(1)}(\theta) \sin(\theta) d\theta \\
	  & + \frac{R^*}{\sin(\alpha^*)} \left(\int^{2 \pi}_0 u^{(1)}(\phi,\pi - \alpha^*) f_k^{(2)}(\phi) d\phi\right) {g^k_m}^{(2)} \\
	  & = \skp{\int^{2\pi}_0 u(\phi,\theta) f_k(\phi) d\phi}{g^k_m}_{\hat{L_2}}.
\end{align*}
For each $k$ the eigenfunctions $(g^k_m)_{m \in \N}$ are complete in $\hat{L_2}$ and so we get
\begin{align*}
	0 = \int^{2\pi}_0 u(\phi,\theta) f_k(\phi) d\phi
\end{align*}
for all $k \in \N$ and almost every $\theta \in [0, \pi - \alpha^*]$. Since $(f_k)_{k \in \N}$ is complete in $L_2([0, 2\pi])$ equipped with the usual $L_2$-inner product, we end up with $u(\phi, \theta) = 0$ almost everywhere. Therefore we arrived at a contradiction to our assumption that $u$ is an eigenfunction. This proves that all eigenfunctions are in the span of functions in product form and justifies the separation ansatz. The last missing ingredient is the proof of the compactness of the embedding $\hat{W^1_2} \hookrightarrow \hat{L_2}$, which we will present now.

\begin{lemma}\label{lem:CompactEmbedding}
The embedding $\hat{W^1_2} \hookrightarrow \hat{L_2}$ is compact.
\end{lemma}
\begin{proof}
To this end let $(u_n)_{n \in \N} \subseteq \hat{W^1_2}$ be bounded. Then we obtain for $t,s \in [0, \pi - \alpha^*]$
\begin{align*}
	|u_n(t) - u_n(s)| & = \left|\int^t_s u_n'(x) dx\right| \leq \left|\int^t_s \sqrt{\sin(x)} u_n'(x) \frac{1}{\sqrt{\sin(x)}} dx\right| \\
							& \leq \int^t_s \left|\sqrt{\sin(x)} u_n'(x)\right| \left|\frac{1}{\sqrt{\sin(x)}}\right| dx \\
							& \leq \underbrace{\left(\int^t_s \sin(x) |u_n'(x)|^2 dx\right)^{\frac{1}{2}}}_{\leq c \text{ since $(u_n)_{n \in \N} \subseteq \hat{W^1_2}$ in bounded}} \left(\int^t_s \frac{1}{\sin(x)} dx\right)^{\frac{1}{2}}.
\end{align*}
Since $\sin(\pi - \alpha^*) > 0$ we can still find a linear function below $\sin(x)$ to continue the estimate as follows
\begin{align}\label{eq:LogarithmicGrowth}
	|u_n(t) - u_n(s)| \leq \hat{c} \left(\int^t_s \frac{1}{x} dx\right)^{\frac{1}{2}} = \hat{c} \left(\ln(t) - \ln(s)\right)^{\frac{1}{2}}.
\end{align}
The fact that the right-hand side is independent of $n$ immediately shows that $(u_n)_{n \in \N}$ is equicontinuous on every compact interval $[a, \pi - \alpha^*] \subseteq (0, \pi - \alpha^*]$. Also on each such compact interval we have the equivalence of the $\hat{W^1_2}$- and $W^1_2$-norms due to $0 < c \leq \sin(\theta) \leq C$. Therefore the usual compact embedding $\hat{W^1_2}([a, \pi - \alpha^*]) \hookrightarrow L_2([a, \pi - \alpha^*])$ holds. Here we define $\hat{W^1_2}([a, \pi - \alpha^*])$ and $L_2([a, \pi - \alpha^*])$ in the same manner as $\hat{W^1_2}$ and $\hat{L_2}$ just the domain for the first component changes to $[a, \pi - \alpha^*]$ instead of $[0, \pi - \alpha^*]$. Hence the bounded sequence $(u_n)_{n \in \N}$ has a subsequence converging in $L_2([a, \pi - \alpha^*])$, which for simplicity shall be called $(u_n)_{n \in \N}$ again. Since $L_2$-convergence implies the pointwise convergence a.e. of a subsequence, we obtain an a.e. pointwise limit of $(u_n)_{n \in \N}$ on $[a, \pi - \alpha^*]$ for each $a > 0$. \\
Using a diagonalisation argument we can show the existence of a subsequence, again denoted by $(u_n)_{n \in \N}$, which converges pointwise on $(0, \pi - \alpha^*]$ to a function that we call $u$. \\
The estimate (\ref{eq:LogarithmicGrowth}) also shows
\begin{align*}
	|u_n(t) - u(t)|^2 \leq c \left(|u_n(\pi - \alpha^*) - u(\pi - \alpha^*)|^2 + |\ln(t)|\right).
\end{align*}
Therefore
\begin{align*}
	\sin(\theta) |u_n(\theta) - u(\theta)|^2 \leq \tilde{c} \sin(\theta) |\ln(\theta)| \qquad \text{ for } \theta \in (0, \pi - \alpha^*)
\end{align*}
and since $\lim_{\theta \rightarrow 0}\limits \sin(\theta) |\ln(\theta)| = 0$ we found a function dominating the sequence $(u_n)_{n \in \N}$, which is still integrable. By dominated convergence theorem we get the $\hat{L_2}$-convergence of $(u_n)_{n \in \N}$. This finally shows that the embedding $\hat{W^1_2} \hookrightarrow \hat{L_2}$ is compact.
\end{proof}

After knowing that all solutions of the homogeneous system (\ref{eq:NullspaceSystem1})-(\ref{eq:NullspaceSystem5}) will be in the span of functions with product structure $\rho(\phi,\theta) = f(\phi) g(\theta)$, we can perform a separation ansatz to transform (\ref{eq:NullspaceSystem1}) with $c = 0$ into equations for $f$ and $g$. Since we are only interested in non-trivial solutions for $\rho$ we can assume $f \not\equiv 0$ and $g \not\equiv 0$. We get
\begin{align}\label{eq:PreSepAnsatz}
	0 & = \frac{1}{\sin(\theta)^2} \rho_{\phi \phi} + \rho_{\theta \theta} + \cot(\theta) \rho_{\theta} + 2 \rho \notag \\
	  & = \frac{1}{\sin(\theta)^2} f'' g + f g'' + \cot(\theta) f g' + 2 f g.
\end{align}
This is equivalent to
\begin{align*}
	-\frac{f''}{f} &= \sin(\theta)^2 \frac{g''}{g} + \sin(\theta) \cos(\theta) \frac{g'}{g} + 2 \sin(\theta)^2,
\end{align*}
where the left hand side is independent of $\theta$ and the right hand side is independent of $\phi$. This justifies
\begin{align}\label{eq:Separationansatz}
	-\frac{f''}{f} = \sin(\theta)^2 \frac{g''}{g} + \sin(\theta) \cos(\theta) \frac{g'}{g} + 2 \sin(\theta)^2 =: \lambda \in \R
\end{align}
This leads to the ODE $f'' + \lambda f = 0$ for $f$ and a second ODE for $g$ that we will examine later.

\begin{rem}
The fact that $f$ or $g$ could be zero in some points does not play any role for (\ref{eq:Separationansatz}). For a fixed $\theta_0 \in [0, \pi - \alpha^*]$ with $g(\theta_0) \neq 0$ we definitely get the ODE $f'' + \lambda f = 0$ on the set $U := \{\phi \in [0, 2\pi] | f(\phi) \neq 0\}$. Assuming that $\phi_0 \in U^c$ we see $f(\phi_0) = 0$ and going back to (\ref{eq:PreSepAnsatz}) we get $0 = \frac{1}{\sin(\theta_0)^2} f''(\phi_0) g(\theta_0)$. Since we assumed $g(\theta_0) \neq 0$, this leads to $f''(\phi_0) = 0$ and therefore $f'' + \lambda f = 0$ is also valid for this $\phi_0$. Interchanging the roles of $f$ and $g$ leads to the same result for $g$.
\end{rem}

The equations (\ref{eq:NullspaceSystem3}) and (\ref{eq:NullspaceSystem4}) translate into boundary conditions for $f$ namely $f(0) = f(2\pi)$ and $f'(0) = f'(2\pi)$. The solution of $f'' + \lambda f = 0$ is obviously given by
\begin{align*}
  f(\phi) = \begin{cases}
    c_1 e^{\sqrt{-\lambda} \phi} + c_2 e^{-\sqrt{-\lambda} \phi}	& \text{ if } \lambda < 0\,, \\
    c_1 + c_2 \phi																& \text{ if } \lambda = 0\,, \\
    c_1 \cos(\sqrt{\lambda} \phi) + c_2 \sin(\sqrt{\lambda} \phi) &
    \text{ if } \lambda > 0\,.
				 \end{cases}
\end{align*}
The boundary conditions leave no non-trivial solution in the case $\lambda \neq k^2$ with $k \in \N$ and in the case $\lambda = k^2$ we end up with the solutions
\begin{align}\label{eq:SolutionF}
	f_k(\phi) = c_1 \cos(k \phi) + c_2 \sin(k \phi) \qquad \text{ with } k \in \N.
\end{align}

Hence from (\ref{eq:Separationansatz}) we get the following ODE for $g$
\begin{align}\label{eq:DGLForG}
	0 = g'' + \cot(\theta) g' + \left(2  - \frac{k^2}{\sin(\theta)^2}\right) g.
\end{align}
So far we have not considered the boundary equations (\ref{eq:NullspaceSystem2}) and (\ref{eq:NullspaceSystem5}). Looking first at (\ref{eq:NullspaceSystem5}) we see
\begin{align*}
	const. = \rho|_{\theta = 0} = f(\phi) g(0),
\end{align*}
which means that either $f(\phi)$ is constant and $\lim_{\theta \downarrow 0}\limits g(\theta)$ exists or otherwise $g(0) = 0$. Since $f$ is constant if $k = 0$, we obtain the condition $g(0) = 0$ for all $k \geq 1$ and ``$\lim_{\theta \downarrow 0}\limits g(\theta)$ exists'' for $k = 0$. \\
Last but not least (\ref{eq:NullspaceSystem2}) transforms into
\begin{align*}
	0 &= \frac{b (1 - k^2)}{R^* \sin(\alpha^*)^2} g(\pi - \alpha^*) - \sin(\alpha^*) g'(\pi - \alpha^*) - \cos(\alpha^*) g(\pi - \alpha^*).
\end{align*}
Hence (\ref{eq:NullspaceSystem2}) and (\ref{eq:NullspaceSystem5}) now read as
\begin{align}\label{eq:BC1ForG}
	0 &= g(0) & &\text{ if $k \geq 1$}, \\ \label{eq:BC2ForG}
	\lim_{\theta \downarrow 0}\limits g(\theta) &\text{ exists } & &\text{ if $k = 0$}, \\ \label{eq:BC3ForG}
	0 &= \left(\frac{b (1 - k^2)}{R^* \sin(\alpha^*)^2} - \cos(\alpha^*)\right) g(\pi - \alpha^*) - \sin(\alpha^*) g'(\pi - \alpha^*) & &\text{ if $k \geq 0$}.
\end{align}
For solving the system (\ref{eq:DGLForG})-(\ref{eq:BC3ForG}) we have to distinguish the cases $k = 0$, $k = 1$ and $k \geq 2$.

1. Case ($k = 0$): Here the general solution of (\ref{eq:DGLForG}) is
\begin{align*}
	g_0(\theta) = c_1 \cos(\theta) + c_2 \left(\frac{1}{2} \cos(\theta) \ln\left(\frac{\cos(\theta) + 1}{\cos(\theta) - 1}\right) - 1\right)
\end{align*}
as one can easily check by differentiation and (\ref{eq:BC1ForG}) does not have to be considered. Due to (\ref{eq:BC2ForG}) we must have $c_2 = 0$ and the solution reduces to $g_0(\theta) = c_1 \cos(\theta)$. The equation (\ref{eq:BC3ForG}) is then given by
\begin{align*}
	0 & = \left(\frac{b}{R^* \sin(\alpha^*)^2} - \cos(\alpha^*)\right) (-c_1 \cos(\alpha^*)) - \sin(\alpha^*) (-c_1 \sin(\alpha^*)) = c_1 \left(1 - \frac{b \cos(\alpha^*)}{R^* \sin(\alpha^*)^2}\right).
\end{align*}
But this means that for $R^* \sin(\alpha^*)^2 \neq b \cos(\alpha^*)$ this equation is only satisfied for $c_1 = 0$ and we do not have any contributing functions from the case $k = 0$. If $R^* \sin(\alpha^*)^2 = b \cos(\alpha^*)$ one can choose any $c_1 \in \R$ and obtain $g_0(\theta) = c_1 \cos(\theta)$ as the solution for $k = 0$. The significance of this special case will be clarified in Remark \ref{rem:CriticalCase} below. \\
2. Case ($k = 1$): Again it is an easy but time-consuming calculation to check that now
\begin{align*}
	g_1(\theta) &= c_1 \sin(\theta) + c_2 \left(-\frac{1}{2} \sin(\theta) \ln\left(\frac{\cos(\theta) + 1}{\cos(\theta) - 1}\right) - \cot(\theta)\right)
\end{align*}
is the general solution of (\ref{eq:DGLForG}). Due to $\frac{1}{2}
\ln\left(\frac{\cos(\theta) + 1}{\cos(\theta) - 1}\right)' =
\frac{-1}{\sin(\theta)}$ we get with L'H\^{o}pital's rule
\begin{align*}
	\lim_{\theta \downarrow 0} \frac{1}{2} \sin(\theta) \ln\left(\frac{\cos(\theta) + 1}{\cos(\theta) - 1}\right)
		= \lim_{\theta \downarrow 0} \frac{\frac{1}{2} \ln\left(\frac{\cos(\theta) + 1}{\cos(\theta) - 1}\right)}{\frac{1}{\sin(\theta)}} 
		= \lim_{\theta \downarrow 0} \frac{\frac{1}{\sin(\theta)}}{\frac{\cos(\theta)}{\sin(\theta)^2}}
		 = \lim_{\theta \downarrow 0} \tan(\theta) = 0
\end{align*}
and thereafter
\begin{align*}
	\lim_{\theta \downarrow 0} g_1(\theta) &= \lim_{\theta \downarrow 0} \left(-c_1 \sin(\theta) + c_2 \left(-\frac{1}{2} \sin(\theta) \ln\left(\frac{\cos(\theta) + 1}{\cos(\theta) - 1}\right) - \cot(\theta)\right)\right) \\
														 &= 0 + 0 - \lim_{\theta \downarrow 0} c_2 \cot(\theta) = -c_2 \lim_{\theta \downarrow 0} \cot(\theta).
\end{align*}
Hence the boundary condition (\ref{eq:BC1ForG}) requires $c_2 = 0$. Therefore the solution is $g_1(\theta) = c_1 \sin(\theta)$. The boundary condition (\ref{eq:BC2ForG}) does not have to be considered and (\ref{eq:BC3ForG}) is now always valid, because
\begin{align*}
	(0 - \cos(\alpha^*)) (c_1 \sin(\alpha^*)) - \sin(\alpha^*) (-c_1 \cos(\alpha^*)) = 0.
\end{align*}
This shows that $g_1(\theta) = c_1 \sin(\theta)$ is the solution for $k = 1$. \\
3. Case ($k \geq 2$): Here we note the close relationship between the operator $\Delta^k_\theta$ from (\ref{eq:ThetaPart}) and the operator given by the right-hand sides of (\ref{eq:DGLForG}) and (\ref{eq:BC3ForG}). We see that a solution of (\ref{eq:DGLForG}) and (\ref{eq:BC3ForG}) would correspond to the eigenvalue $0$ for the operator (\ref{eq:ThetaPart}). Therefore it is enough to show that there is no eigenvalue $0$ for $k \geq 2$ of $\Delta^k_\theta$. We assume that we would have an eigenfunction $g$ of $\Delta^k_\theta$ corresponding to the eigenvalue $0$. Using (\ref{eq:BilinearDTheta}) we would obtain
\begin{align}\label{eq:PositiveEVs}
	0 & = \skp{\Delta^k_\theta g}{g}_{\hat{L_2}} = B(g,g) \notag \\
	  & = \int^{\pi - \alpha^*}_0 (g^{(1)}_\theta)^2 \sin(\theta) - \left(2 - \frac{k^2}{\sin(\theta)^2}\right) (g^{(1)})^2 \sin(\theta) d\theta + \left(\cos(\alpha^*) - \frac{b (1 - k^2)}{R^* \sin(\alpha^*)^2}\right) (g^{(2)})^2 \notag \\
	  & = \int^{\pi - \alpha^*}_0 (g^{(1)}_\theta)^2 \sin(\theta) + \underbrace{\left(\frac{k^2}{\sin(\theta)^2} - 2\right)}_{\geq \frac{4}{1} - 2 = 2} (g^{(1)})^2 \sin(\theta) d\theta + \underbrace{\left(\frac{b (k^2 - 1)}{R^* \sin(\alpha^*)^2} + \cos(\alpha^*)\right)}_{\geq \frac{3b}{R^* \sin(\alpha^*)^2} + \cos(\alpha^*)} (g^{(2)})^2 \notag \\
	  & \geq \int^{\pi - \alpha^*}_0 (g^{(1)}_\theta)^2 \sin(\theta) + 2 (g^{(1)})^2 \sin(\theta) d\theta + \left(\frac{3b}{R^* \sin(\alpha^*)^2} + \cos(\alpha^*)\right) (g^{(2)})^2.
\end{align}
For $b > C_{crit} := -\frac{1}{3} R^* \sin(\alpha^*)^2 \cos(\alpha^*) = -\frac{1}{3} H^* \sin(\alpha^*)^2$ this is a contradiction, because the last term is strictly positive. Therefore we do not get any additional solutions from the cases $k \geq 2$.

\begin{rem}\label{rem:RangeOfB}
(i) If $\cos(\alpha^*) \geq 0$, or equivalently $H^* \geq 0$, the critical constant $C_{crit}$ is negative or zero and hence $b > C_{crit}$ is always satisfied. Therefore we have no nullspace elements for $k \geq 2$ in this case. \\
(ii) What we have done in the considerations for $k \geq 2$ above is actually much more valuable than it seems at the first glance. If we modify the calculations a little and assume that $g$ is an eigenfunction of $\Delta^k_\theta$ corresponding to an arbitrary eigenvalue $\lambda$. Then (\ref{eq:PositiveEVs}) reads as
\begin{align*}
	\lambda \skp{g}{g}_{\hat{L_2}} = \skp{\Delta^k_\theta g}{g}_{\hat{L_2}} = B(g,g) > 0.
\end{align*}
Yet, this shows that all eigenvalues $\mu_k$ of $\Delta^k_\theta$ and due to (\ref{eq:EigenpairsLaplaceB}) thereby also the eigenvalues of $\Delta^B$ are all positive for $k \geq 2$.
\end{rem}

Now we want to close the gap in our argument occurring from the case $R^* \sin(\alpha^*)^2 = b \cos(\alpha^*)$.

\begin{lemma}\label{lem:CriticalCase}
In the case $R^* \sin(\alpha^*)^2 = b \cos(\alpha^*)$ the system (\ref{eq:NullspaceSystem1})-(\ref{eq:NullspaceSystem5}) has no solution if $c \neq 0$.
\end{lemma}
\begin{proof}
We note that it suffices to consider $\cos(\alpha^*) > 0$, since $R^* \sin(\alpha^*)^2 = b \cos(\alpha^*)$ can not occur if $\cos(\alpha^*) \leq 0$. Moreover, we can ignore $b > C_{crit}$ in this case, since $C_{crit} < 0$. Then we rewrite (\ref{eq:NullspaceSystem1})-(\ref{eq:NullspaceSystem5}) for this particular situation and get
\begin{align}\label{eq:CriticalSystem1}
	c &= \frac{1}{\sin(\theta)^2} \rho_{\phi \phi} + \rho_{\theta \theta} + \cot(\theta) \rho_{\theta} + 2 \rho & &\quad \text{in } (0, 2\pi) \times (0, \pi - \alpha^*), \\ \label{eq:CriticalSystem2}
	0 &= \frac{1}{\cos(\alpha^*)} (\rho_{\phi \phi} + \rho) - \sin(\alpha^*) \rho_\theta - \cos(\alpha^*) \rho & &\quad \text{on } [0, 2\pi] \times \{\pi - \alpha^*\}, \\ \label{eq:CriticalSystem3}
	0 &= \rho|_{\phi = 0} - \rho|_{\phi = 2\pi} & &\quad \text{on } [0, \pi - \alpha^*], \\ \label{eq:CriticalSystem4}
	0 &= \rho_\phi|_{\phi = 0} - \rho_\phi|_{\phi = 2\pi} & &\quad \text{on } [0, \pi - \alpha^*], \\ \label{eq:CriticalSystem5}
	const. &= \rho|_{\theta = 0} & &\quad \text{on } [0, 2\pi].
\end{align}
The ideas for this proof are taken from \cite{Nar02}. The periodicity from (\ref{eq:CriticalSystem3})-(\ref{eq:CriticalSystem4}) in $\phi$ justifies an ansatz of the form
\begin{align*}
	\rho(\phi,\theta) = \sum^\infty_{m = -\infty} \hat{\rho}_m(\theta) e^{i m \phi}.
\end{align*}
Using this in (\ref{eq:CriticalSystem1}) we obtain
\begin{align*}
	\sum^\infty_{m = -\infty} c \delta_{m0} e^{i m \phi}
		& = c = \left(\frac{1}{\sin(\theta)^2} \d_{\phi \phi} + \d_{\theta \theta} + \cot(\theta) \d_{\theta} + 2\right) \rho \\
		& = \sum^\infty_{m = -\infty} \left(\hat{\rho}_m'' + \cot(\theta) \hat{\rho}_m' + \left(2 - \frac{m^2}{\sin(\theta)^2}\right) \hat{\rho}_m\right) e^{i m \phi},
\end{align*}
where $\delta_{ij}$ denotes the Kronecker delta. Interchanging the operator with the summation as well as the convergence of the sum is justified by the smoothness of $\rho$ on $[0, 2\pi] \times [0, \pi - \alpha^*]$. The same ansatz in (\ref{eq:CriticalSystem2}) and (\ref{eq:CriticalSystem5}) gives
\begin{align*}
	0 = \sum^\infty_{m = -\infty} \left(\left(\frac{1 - m^2}{\cos(\alpha^*)} - \cos(\alpha^*)\right) \hat{\rho}_m - \sin(\alpha^*) \hat{\rho}_m'\right) e^{i m \phi}
\end{align*}
and $const. = \sum^\infty_{m = -\infty}\limits \hat{\rho}_m(0) e^{i m \phi}$, respectively. Since the Fourier series is unique we can equate the coefficients and this leads to the following two ODEs
\begin{align}\label{eq:ODEFor0}
	c & = \hat{\rho}_0''(\theta) + \cot(\theta) \hat{\rho}_0'(\theta) + 2 \hat{\rho}_0(\theta), \\ \label{eq:BCFor0}
	0 & = \sin(\alpha^*) \hat{\rho}_0(\pi - \alpha^*) - \cos(\alpha^*) \hat{\rho}_0'(\pi - \alpha^*), \\ \label{eq:ICFor0}
	\lim_{\theta \downarrow 0} \hat{\rho}_0(\theta) &\text{ exists}	
\end{align}
and
\begin{align}\label{eq:ODEForM}
	0 & = \hat{\rho}_m''(\theta) + \cot(\theta) \hat{\rho}_m'(\theta) + \left(2 - \frac{m^2}{\sin(\theta)^2}\right) \hat{\rho}_m(\theta), \\ \label{eq:BCForM}
	0 & = \left(\frac{1 - m^2}{\cos(\alpha^*)} - \cos(\alpha^*)\right) \hat{\rho}_m(\pi - \alpha^*) - \sin(\alpha^*) \hat{\rho}_m'(\pi - \alpha^*), \\ \label{eq:ICForM}
	0 & = \hat{\rho}_m(0)
\end{align}
for $m \neq 0$. We start by investigating the second system. Assuming that we have a solution for it, we would get
\begin{align*}
	\frac{m^2}{\sin(\theta)^2} \hat{\rho}_m = \frac{1}{\sin(\theta)} \left(\sin(\theta) \hat{\rho}_m'\right)' + 2 \hat{\rho}_m.
\end{align*}
Multiplying with $\sin(\theta) \hat{\rho}_m$ and integrating over $[0, \pi - \alpha^*]$ gives
\begin{align*}
	m^2 \int^{\pi - \alpha^*}_0 \frac{1}{\sin(\theta)} \hat{\rho}_m^2 d\theta
		& = \int^{\pi - \alpha^*}_0 \left(\sin(\theta) \hat{\rho}_m'\right)' \hat{\rho}_m + 2 \hat{\rho}_m^2 \sin(\theta) d\theta \\
		& = \left[\sin(\theta) \hat{\rho}_m' \hat{\rho}_m\right]^{\pi - \alpha^*}_0 - \int^{\pi - \alpha^*}_0 \hat{\rho}_m'^2 \sin(\theta) d\theta + 2 \int^{\pi - \alpha^*}_0 \hat{\rho}_m^2 \sin(\theta) d\theta \\
		& \leq \sin(\alpha^*) \hat{\rho}_m'(\pi - \alpha^*) \hat{\rho}_m(\pi - \alpha^*) + 2 \int^{\pi - \alpha^*}_0 \hat{\rho}_m^2 \sin(\theta) d\theta \\
		& = \underbrace{\left(\frac{1 - m^2}{\cos(\alpha^*)} - \cos(\alpha^*)\right)}_{< 0} \hat{\rho}_m(\pi - \alpha^*)^2 + 2 \int^{\pi - \alpha^*}_0 \hat{\rho}_m^2 \sin(\theta) d\theta \\
		& \leq 2 \int^{\pi - \alpha^*}_0 \hat{\rho}_m^2 \sin(\theta) d\theta.
\end{align*}
Yet, this leaves us with an upper bound for $m^2$, namely
\begin{align*}
	m^2 \leq 2 \frac{\int^{\pi - \alpha^*}_0\limits \hat{\rho}_m^2 \sin(\theta) d\theta}{\int^{\pi - \alpha^*}_0\limits \frac{1}{\sin(\theta)} \hat{\rho}_m^2 d\theta} \leq 2.
\end{align*}
This shows that the system (\ref{eq:ODEForM})-(\ref{eq:ICForM}) only has to be considered for $m^2 = 1$. This reduces (\ref{eq:ODEForM})-(\ref{eq:ICForM}) to
\begin{align}\label{eq:ODEFor1}
	0 & = \hat{\rho}_1''(\theta) + \cot(\theta) \hat{\rho}_1'(\theta) + \left(2 - \frac{1}{\sin(\theta)^2}\right) \hat{\rho}_1(\theta), \\ \label{eq:BCFor1}
	0 & = -\cos(\alpha^*) \hat{\rho}_1(\pi - \alpha^*) - \sin(\alpha^*) \hat{\rho}_1'(\pi - \alpha^*), \\ \label{eq:ICFor1}
	0 & = \hat{\rho}_1(0).
\end{align}
The general solution of (\ref{eq:ODEFor1}) is given by
\begin{align*}
	\hat{\rho}_1(\theta) = -c_1\sin(\theta) + c_2 \left(-\frac{1}{2} \sin(\theta) \ln\left(\frac{1 + \cos(\theta)}{1 - \cos(\theta)}\right) - \cot(\theta)\right).
\end{align*}
For $\hat{\rho}_1$ to solve (\ref{eq:ICFor1}) we require $c_2 = 0$ and (\ref{eq:BCFor1}) is always satisfied. Hence $\hat{\rho}_1(\theta) = -c_1\sin(\theta)$ is the complete solution of (\ref{eq:ODEFor1})-(\ref{eq:ICFor1}). \\
Now we consider the system (\ref{eq:ODEFor0})-(\ref{eq:ICFor0}). The general solution of (\ref{eq:ODEFor0}) is given by
\begin{align*}
	\hat{\rho}_0(\theta) = \frac{c}{2} + c_1 \cos(\theta) + c_2 \left(\frac{1}{2} \cos(\theta) \ln\left(\frac{1 + \cos(\theta)}{1 - \cos(\theta)}\right) - 1\right).
\end{align*}
If $c_2 \neq 0$ the function would have a singularity in $\theta = 0$, which makes it necessary for (\ref{eq:ICFor0}) that $c_2 = 0$. Therefore we know that so far the solution $\hat{\rho}_0$ is of the form
\begin{align*}
	\hat{\rho}_0(\theta) = \frac{c}{2} + c_1 \cos(\theta).
\end{align*}
The boundary condition (\ref{eq:BCFor0}) is only satisfied for $c = 0$ as one can see from
\begin{align*}
	0 = \sin(\alpha^*) \hat{\rho}_0(\pi - \alpha^*) - \cos(\alpha^*) \hat{\rho}_0'(\pi - \alpha^*) = \sin(\alpha^*) \frac{c}{2}.
\end{align*}
This is the contradiction that we are looking for.
\end{proof}

\begin{rem}\label{rem:CriticalCase}
(i) We continue the considerations from the previous proof one step further: Since $e^{i \phi}$ and $e^{-i \phi}$ can be transformed into $\sin(\phi)$ and $\cos(\phi)$, we end up with the solution
\begin{align*}
	\rho(\phi,\theta) = c_1 \cos(\theta) + c_2 \cos(\phi) \sin(\theta) + c_3 \sin(\phi) \sin(\theta),
\end{align*}
which is exactly what we have obtained in the cases $k = 0$, $k = 1$ and $k \geq 2$ above. \\
(ii) Lemma \ref{lem:CriticalCase} explains why we found for $R^* \sin(\alpha^*)^2 = b \cos(\alpha^*)$ an additional function while considering the case $k = 0$ above. This particular function compensates the missing special solution if $R^* \sin(\alpha^*)^2 = b \cos(\alpha^*)$, so that we always find three linearly independent functions in $\mathcal{N}(A_0)$ if we restrict ourselves to $b > C_{crit}$.
\end{rem}

If $b > C_{crit}$, then
\begin{align}\label{eq:FinalRho}
	\rho(\phi,\theta) = \begin{cases}
								  \begin{aligned}
									  &c_1 (1 + c_\alpha \cos(\theta)) \\
									  &\hspace*{4mm} + c_2 \cos(\phi) \sin(\theta) + c_3 \sin(\phi) \sin(\theta)
								  \end{aligned} & \text{ if } R^* \sin(\alpha^*)^2 \neq b \cos(\alpha^*), \\
								  \begin{aligned}
									  &c_1 \cos(\theta) \\
									  &\hspace*{4mm} + c_2 \cos(\phi) \sin(\theta) + c_3 \sin(\phi) \sin(\theta)
								  \end{aligned} & \text{ if } R^* \sin(\alpha^*)^2 = b \cos(\alpha^*)
							  \end{cases}
\end{align}
is the full solution to the inhomogeneous system (\ref{eq:NullspaceSystem1})-(\ref{eq:NullspaceSystem5}).

Transforming (\ref{eq:FinalRho}) back to the usual $x$-$y$-$z$-coordinates one can see that the last two linearly independent summands that (\ref{eq:FinalRho}) consists of, are the expected shifts in $x$- and $y$-direction. In fact, using (\ref{eq:SphericalCoordinates}) we have
\begin{align*}
	\sin(\phi) = \frac{x}{\sqrt{x^2 + y^2}}, \qquad \cos(\phi) = \frac{y}{\sqrt{x^2 + y^2}} \qquad \text{ and } \qquad
	\sin(\theta) = \frac{1}{R^*} \sqrt{x^2 + y^2},
\end{align*}
which shows
\begin{align}\label{eq:XShift}
	\tilde{\rho}_1(x,y) & = \frac{x}{\sqrt{x^2 + y^2}} \frac{1}{R^*} \sqrt{x^2 + y^2} = \frac{x}{R^*}, \\ \label{eq:YShift}
	\tilde{\rho}_2(x,y) & = \frac{y}{\sqrt{x^2 + y^2}} \frac{1}{R^*} \sqrt{x^2 + y^2} = \frac{y}{R^*}.
\end{align}
The first linearly independent summand in (\ref{eq:FinalRho}) transforms using
\begin{align*}
	\cos(\theta) = \frac{z(x,y) - H^*}{R^*}
\end{align*}
into
\begin{align}\label{eq:ZShiftExp}
	\tilde{\rho}_0(x,y) = \begin{dcases}
									 \frac{R^* - c_\alpha H^*}{R^*} + c_\alpha \frac{z(x,y)}{R^*} & \text{ if } R^* \sin(\alpha^*)^2 \neq b \cos(\alpha^*), \\
									 \frac{z(x,y)}{R^*} - \frac{H^*}{R^*} & \text{ if } R^* \sin(\alpha^*)^2 = b \cos(\alpha^*).
								 \end{dcases}
\end{align}
This is a combination of a radial expansion and a shift in $z$-direction. Defining
\begin{align}\label{eq:NullspaceElements}
	v_i := \begin{pmatrix}
				 \tilde{\rho}_i \\
				 \tilde{\rho}_i|_{\d \Gamma^*}
			 \end{pmatrix} \in V \cap X_1 \qquad \text{ for each } i \in \{0,1,2\}
\end{align}
we have $\mathcal{N}(A_0) = \spn \{v_0, v_1, v_2\}$ and especially $\dim(\mathcal{N}(A_0)) = 3$ whenever $b > C_{crit}$.

Since the $3$-dimensionality of $\mathcal{N}(A_0)$ will play a crucial role in all the considerations to follow, we assume from now on
\begin{align}\label{eq:CriticalAssuption}
	b > C_{crit} = -\frac{1}{3} R^* \sin(\alpha^*)^2 \cos(\alpha^*) = -\frac{H^*}{3} \sin(\alpha^*)^2.
\end{align}

Now that we studied $A_0$ and its nullspace intensively, we still can not start checking the assumptions (a)-(d) from Theorem \ref{thm:GPLStability}. For proving assumption (a) we first have to investigate the solvability of
\begin{align}\label{eq:EllipticPDE1}
	-\Delta_{\Gamma^*} v^{(1)} - |\sigma^*|^2 v^{(1)} + \mint_{\Gamma^*} \Delta_{\Gamma^*} v^{(1)} + |\sigma^*|^2 v^{(1)} \dH^2 & = f^{(1)} & &\text{in } \Gamma^*, \\ \label{eq:EllipticPDE2}
	\sin(\alpha^*)^2 (n_{\d \Gamma^*} \cdot \nabla_{\Gamma^*} v^{(1)}) + \frac{\sin(\alpha^*) \cos(\alpha^*)}{R^*} v^{(1)} & & & \notag \\
	- b \sin(\alpha^*) v^{(2)}_{\sigma\sigma} - \frac{b}{{R^*}^2 \sin(\alpha^*)} v^{(2)} & = f^{(2)} & &\text{on } \d \Gamma^*
\end{align}
for a right-hand side $f = (f^{(1)}, f^{(2)})$.

First we will need the notion of a weak solution and later use semigroup arguments to show higher regularity of these solutions.

\begin{defi}[Weak solution]
We call
\begin{align*}
	u = (u^{(1)},u^{(2)}) \in H := \left\{W^1_2(\Gamma^*) \times W^1_2(\d \Gamma^*) \left| \, u^{(1)}|_{\d \Gamma^*} = u^{(2)}, \, \mint_{\Gamma^*} u^{(1)} \dH^2 = 0\right.\right\}
\end{align*}
a weak solution of (\ref{eq:EllipticPDE1})-(\ref{eq:EllipticPDE2}) for $f = (f^{(1)}, f^{(2)}) \in \tilde{L_2}$ with
\begin{align*}
	\tilde{L_2} & := \left\{f \in L_2(\Gamma^*) \times L_2(\d \Gamma^*) \left| \, \mint_{\Gamma^*} f^{(1)} \dH^2 = 0\right.\right\}, \\
	\skp{f}{g}_{\tilde{L_2}} & := \int_{\Gamma^*} f^{(1)} g^{(1)} \dH^2 + \int_{\d \Gamma^*} \frac{1}{\sin(\alpha^*)^2} f^{(2)} g^{(2)} \dH^1
\end{align*}
if we have
\begin{align}\label{eq:WeakSolution}
	&\int_{\Gamma^*} \nabla_{\Gamma^*} u^{(1)} \cdot \nabla_{\Gamma^*} v^{(1)} \dH^2 - \int_{\Gamma^*} |\sigma^*|^2 u^{(1)} v^{(1)} \dH^2 \notag \\
	&+ \int_{\d \Gamma^*} \frac{b}{\sin(\alpha^*)} u^{(2)}_\sigma v^{(2)}_\sigma \dH^1 + \int_{\d \Gamma^*} \left(\frac{\cot(\alpha^*)}{R^*} - \frac{b}{R^* \sin(\alpha^*)^3}\right) u^{(2)} v^{(2)} \dH^1 \notag \\
	& \qquad = \int_{\Gamma^*} f^{(1)} v^{(1)} \dH^2 + \int_{\d \Gamma^*} \frac{1}{\sin(\alpha^*)^2} f^{(2)} v^{(2)} \dH^1
\end{align}
for all $v \in H$.
\end{defi}

This definition is motivated by the fact that a solution $u \in C^2$ of (\ref{eq:EllipticPDE1})-(\ref{eq:EllipticPDE2}) satisfies (\ref{eq:WeakSolution}). For using the Lemma of Lax-Milgram we define the bilinear form $B: H \times H \longrightarrow \R$ and the functional $F: H \longrightarrow \R$ by
\begin{align*}
	B(u,v) & := \int_{\Gamma^*} \nabla_{\Gamma^*} u^{(1)} \cdot \nabla_{\Gamma^*} v^{(1)} \dH^2 - \int_{\Gamma^*} |\sigma^*|^2 u^{(1)} v^{(1)} \dH^2 \\
			 & + \int_{\d \Gamma^*} \frac{b}{\sin(\alpha^*)} u^{(2)}_\sigma v^{(2)}_\sigma \dH^1 + \int_{\d \Gamma^*} \left(\frac{\cot(\alpha^*)}{R^*} - \frac{b}{{R^*}^2 \sin(\alpha^*)^3}\right) u^{(2)} v^{(2)} \dH^1, \\
	F(v) & := \skp{f}{v}_{\tilde{L_2}}.
\end{align*}
$B$ and $F$ are bounded, which we can see by straight forward
estimates and usage of H\"older's inequality. Moreover, we have the
energy identity
\begin{align*}
	B(u,u) & = \norm{\nabla_{\Gamma^*} u^{(1)}}_{L_2(\Gamma^*)}^2 - \frac{2}{{R^*}^2} \norm{u^{(1)}}_{L_2(\Gamma^*)}^2 \\
			 & + \frac{b}{\sin(\alpha^*)} \norm{u^{(2)}_\sigma}_{L_2(\d \Gamma^*)}^2 + \underbrace{\left(\frac{\cot(\alpha^*)}{{R^*}^2} - \frac{b}{{R^*}^2 \sin(\alpha^*)^3}\right)}_{=: \hat{c}} \norm{u^{(2)}}_{L_2(\d \Gamma^*)}^2.
\end{align*}
If $\hat{c} \geq 0$, we can drop the last summand to obtain
\begin{align*}
	B(u,u) + c_7 \norm{u^{(1)}}_{L_2(\Gamma^*)}^2 \geq \norm{\nabla_{\Gamma^*} u^{(1)}}_{L_2(\Gamma^*)}^2 + c_8 \norm{u^{(2)}_\sigma}_{L_2(\d \Gamma^*)}^2
\end{align*}
and thus we see
\begin{align*}
	B(u,u) + C \norm{u}_{\tilde{L_2}}^2 \geq c_9 \left(\norm{u^{(1)}}_{W^1_2(\Gamma^*)}^2 + \norm{u^{(2)}}_{W^1_2(\d \Gamma^*)}^2\right) \geq c \norm{u}_H^2
\end{align*}
for some $C, c > 0$. Should $\hat{c} < 0$ hold, then we can absorb this last summand into $\norm{u}_{\tilde{L_2}}^2$ on the left-hand side and still arrive at the inequality $B(u,u) + C \norm{u}_{\tilde{L_2}}^2 \geq c \norm{u}_H^2$. \\
This shows that for $\mu \geq C$ the modified bilinear form
\begin{align*}
	B_\mu: H \times H \longrightarrow \R: (u,v) \longmapsto B_\mu(u,v) := B(u,v) + \mu \skp{u}{v}_{\tilde{L_2}}
\end{align*}
satisfies all the assumptions that are necessary to use the Lemma of Lax-Milgram. Therefore we know that for each $f \in \tilde{L_2}$ there exists a unique weak solution $u \in H$ of the modified equation
\begin{align}\label{eq:EllipticPDEMu1}
	-\Delta_{\Gamma^*} v^{(1)} - |\sigma^*|^2 v^{(1)} & & & \notag \\
	+ \mint_{\Gamma^*} \Delta_{\Gamma^*} v^{(1)} + |\sigma^*|^2 v^{(1)} \dH^2 + \mu v^{(1)} =: (L_\mu u)^{(1)} & = f^{(1)} & &\text{in } \Gamma^*, \\ \label{eq:EllipticPDEMu2}
	\sin(\alpha^*)^2 (n_{\d \Gamma^*} \cdot \nabla_{\Gamma^*} v^{(1)}) + \frac{\sin(\alpha^*) \cos(\alpha^*)}{R^*} v^{(1)} & & & \notag \\
	- b \sin(\alpha^*) v^{(2)}_{\sigma\sigma} - \frac{b}{{R^*}^2 \sin(\alpha^*)} v^{(2)} + \mu v^{(2)} =: (L_\mu u)^{(2)} & = f^{(2)} & &\text{on } \d \Gamma^*.
\end{align}
This unique solution $u$ shall be denoted by $u = L_\mu^{-1} f$. A weak solution $u \in H$ of the original problem (\ref{eq:EllipticPDE1})-(\ref{eq:EllipticPDE2}) for a right-hand side $f \in \tilde{L_2}$ is equivalent to a weak solution of (\ref{eq:EllipticPDEMu1})-(\ref{eq:EllipticPDEMu2}) with a right-hand side $\mu u + f$, i.e. a $u \in H$ satisfying
\begin{align*}
	B_\mu(u,v) = \skp{\mu u + f}{v}_{\tilde{L_2}} \qquad \forall \, v \in H.
\end{align*}
Using the weak solvability we obtain $u = L_\mu^{-1}(\mu u + f)$, which can be transformed into $(\id - K)u = g$ with $g := L_\mu^{-1} f$ and $K := \mu L_\mu^{-1}$. Note that $K: \tilde{L_2} \rightarrow H$ is bounded due to
\begin{align*}
	c \norm{u}_H^2 \leq B_\mu(u,u) = \skp{g}{u}_{\tilde{L_2}} \leq \norm{g}_{\tilde{L_2}} \norm{u}_{\tilde{L_2}} \leq \norm{g}_{\tilde{L_2}} \norm{u}_H,
\end{align*}
which shows
\begin{align*}
	c \norm{K g}_H = c \mu \norm{L_\mu^{-1} g}_H = c \mu \norm{u}_H \leq \mu \norm{g}_{\tilde{L_2}}.
\end{align*}
Regarding $K$ as an operator $K: \tilde{L_2} \longrightarrow H \hookrightarrow \tilde{L_2}$ it is compact as a composition of a bounded operator and the compact embedding $H \hookrightarrow \tilde{L_2}$. Fredholm theory gives that $u - Ku = g$ has a solution if and only if $\skp{g}{v}_{\tilde{L_2}} = 0$ for all $v \in H$ with $v - K^* v = 0$. This condition can be rewritten as $\skp{f}{v}_{\tilde{L_2}} = 0$ for all $v \in H$ with $v = K^* v$, because of
\begin{align*}
	0 = \skp{g}{v}_{\tilde{L_2}} = \skp{L_\mu^{-1} f}{v}_{\tilde{L_2}} = \frac{1}{\mu} \skp{K f}{v}_{\tilde{L_2}} = \frac{1}{\mu} \skp{f}{K^* v}_{\tilde{L_2}} = \frac{1}{\mu} \skp{f}{v}_{\tilde{L_2}}.
\end{align*}
The condition $v - K^*v = 0$, however, is equivalent to $B(v,u) = 0$ for all $u \in H$ due to the symmetry of $B$ on $H$. Note that $B(u,v) = 0$ for all $u \in H$ is the same as finding solutions of
\begin{align*}
	-\Delta_{\Gamma^*} v^{(1)} - |\sigma^*|^2 v^{(1)} &= const. & &\text{in } \Gamma^*, \\
	\sin(\alpha^*)^2 (n_{\d \Gamma^*} \cdot \nabla_{\Gamma^*} v^{(1)}) + \frac{\sin(\alpha^*) \cos(\alpha^*)}{R^*} v^{(1)} & & & \\
	- b \sin(\alpha^*) v^{(2)}_{\sigma\sigma} - \frac{b}{{R^*}^2 \sin(\alpha^*)} v^{(2)} &= 0 & &\text{on } \d \Gamma^*, \\
	\int_{\Gamma^*} v^{(1)} \dH^2 &= 0, & &
\end{align*}
which we already did as we determined $\mathcal{N}(A_0)$ and found these equations to be satisfied exactly for $v_1$ and $v_2$ from (\ref{eq:NullspaceElements}). The nullspace element $v_0$ is omitted, since its first component is not mean value free as required for $H$. Summing up we proved (\ref{eq:EllipticPDE1})-(\ref{eq:EllipticPDE2}) has a weak solution $u \in H$ if and only if $f \in \tilde{L_2}$ satisfies $\skp{f}{v_1}_{\tilde{L_2}} = \skp{f}{v_2}_{\tilde{L_2}} = 0$. \\
The next step is to show that the weak solution is actually a strong solution. Let $f \in X_0$ such that $\int_{\Gamma^*}\limits f^{(1)} \dH^2 = 0$ and $\skp{f}{v_1}_{\tilde{L_2}} = \skp{f}{v_2}_{\tilde{L_2}} = 0$. Then we know by Theorem \ref{thm:Semigroup2} that $-A_0$ generates an analytic semigroup and hence there exists some $\mu_0 > 0$ such that $\mu_0 u + A_0 u = f$ has a unique solution $u \in X_1$. The weak solution $u_w \in H$ of $A_0 u_w = f$ also solves $\mu_0 u_w + A_0 u_w = \mu_0 u_w + f =: \hat{f}$ weakly. We see that {$\hat{f} \in X_0$ if $3<p\leq 4$} due to {$u_w \in H \subseteq X_0$ in this case} and the choice of $f$. Thus we obtain another $u_s \in X_1$, which also solves $\mu_0 u_s + A_0 u_s = \hat{f}$. {In the case $p>4$ we obtain the same conclusion by using the previous argument for some $3<\tilde{p}\leq 4$ and using bootstraping once. In both cases we obtain that,} since this $u_s$ is also a weak solution and hence is unique, it has to coincide with $u_w$. Thus the solution $u_w$ of $A_0 u_w = f$ is not only in $H$, but even an element of $X_1 \cap H$. So far we have seen that (\ref{eq:EllipticPDE1})-(\ref{eq:EllipticPDE2}) has a solution $u \in X_1$ with $\int_{\Gamma^*}\limits u^{(1)} \dH^2 = 0$ for all $f \in X_0$ that satisfies $\int_{\Gamma^*}\limits f^{(1)} \dH^2 = 0$ and $\skp{f}{v_1}_{\tilde{L_2}} = \skp{f}{v_2}_{\tilde{L_2}} = 0$.

These considerations regarding the nullspace and the solvability of (\ref{eq:EllipticPDE1})-(\ref{eq:EllipticPDE2}) put us into the position of finally start proving the assumptions (a)-(d) from Theorem \ref{thm:GPLStability}. 

We turn our attention to assumption (a) and prove it in the upcoming lemma.

\begin{lemma}\label{lem:Manifold}
Near $v^* \equiv 0$ the set of equilibria $\mathcal{E}$ of (\ref{eq:Flow1})-(\ref{eq:Flow2}) is a $C^1$-manifold in $X_1$.
\end{lemma}
\begin{proof}
We will enclose the set of equilibria $\mathcal{E}$ between a smaller set $\tilde{\mathcal{E}}$ and a bigger set $\hat{\mathcal{E}}$ that are $3$-dimensional $C^1$-manifolds and hence $\mathcal{E}$ is a $3$-dimensional $C^1$-manifold as well. The arguments will rely on Theorem 4.B in \cite{Zei85}. To this end define
\begin{align*}
	X := \R^3, \qquad Y := \modulo{X_1}{\mathcal{N}(A_0)} \qquad \text{ and } \qquad Z := \left\{v \in X_0 \left| \int_{\Gamma^*} v^{(1)} \dH^2 = 0\right.\right\}.
\end{align*}
Then $X$ and $Z$ are Banach spaces and $Y$ as well, since $\mathcal{N}(A_0)$ is finite dimensional and hence closed. We consider the function
\begin{align*}
	F: X \times Y \longrightarrow Z: (t_0,t_1,t_2,w) \longmapsto \begin{pmatrix}
																						 H_\Gamma(v^{(1)}) - \bar{H}(v^{(1)}) \\
																						 a + b \kappa_{\d D}(v^{(2)}) + \skp{n_\Gamma(v^{(1)})}{n_D}
																					 \end{pmatrix},
\end{align*}
where $v = (v^{(1)}, v^{(2)})^T$ shall be given by $v := t_0 v_0 + t_1 v_1 + t_2 v_2 + w$ with $w \in \mathcal{N}(A_0)^\bot$. Then the set of equilibria as given in (\ref{eq:Equilibria}) can be written as $\mathcal{E} = \{v \in V \cap X_1 \mid F(v) = 0\}$. We use the orthogonal projection $P: X_0 \longrightarrow \spn \{v_1, v_2\}^\perp$, where the orthogonal complement has to be understood with respect to the $\tilde{L_2}$-inner product, to define
\begin{align*}
	\hat{\mathcal{E}} := \{v \in V \cap X_1 \mid PF(v) = 0\}.
\end{align*}
Then trivially $\mathcal{E} \subseteq \hat{\mathcal{E}}$ and $PF$ maps as follows
\begin{align*}
	PF: X \times Y \longrightarrow \spn \{v_1, v_2\}^\perp \cap Z \subseteq X_0: (t_0,t_1,t_2,w) \longmapsto PF(v)
\end{align*}
for $v = t_0 v_0 + t_1 v_1 + t_2 v_2 + w$. In $\mathbb{O} := (0,0,0,0) \in X \times Y$ the partial derivative of $F$ with respect to $w$, which corresponds to the linearization operator $-A_0$, is given by (\ref{eq:InteriorForSSC}), (\ref{eq:BCForSSC})-(\ref{eq:SSCQuantities2}) as
\begin{align*}
	\left(D_wF(\mathbb{O})(v)\right)^{(1)} = \Delta_{\Gamma^*} v^{(1)} + |\sigma^*|^2 v^{(1)} - \mint_{\Gamma^*}\limits \Delta_{\Gamma^*} v^{(1)} + |\sigma^*|^2 v^{(1)} \dH^2
\end{align*}
and
\begin{align*}
	\left(D_wF(\mathbb{O})(v)\right)^{(2)}
		& = -\sin(\alpha^*)^2 (n_{\d \Gamma^*} \cdot \nabla_{\Gamma^*} v^{(1)}) - \frac{\sin(\alpha^*) \cos(\alpha^*)}{R^*} v^{(1)} \\
		& + b \sin(\alpha^*) v^{(2)}_{\sigma\sigma} + \frac{b}{{R^*}^2 \sin(\alpha^*)} v^{(2)}.
\end{align*}
Now we will show that
\begin{align*}
	D_w(PF)(\mathbb{O}): \modulo{X_1}{\mathcal{N}(A_0)} \longrightarrow \spn \{v_1, v_2\}^\perp \cap Z: w \longmapsto D_w(PF)(w)
\end{align*}
is bijective. First remark that $D_w(PF) = PD_wF$ since $P$ is linear. The injectivity can be seen from
\begin{align*}
	D_w(PF)(\mathbb{O}) w = 0 \quad &\Leftrightarrow \quad P(D_wF(\mathbb{O})w) = 0 & &\Leftrightarrow \quad PA_0w = 0 \\
		&\Leftrightarrow \quad A_0w \in \mathcal{N}(A_0) & &\Leftrightarrow \quad w \in \mathcal{N}(A_0^2) = \mathcal{N}(A_0) \\
		&\Leftrightarrow \quad w = 0 \in \modulo{X_1}{\mathcal{N}(A_0)},
\end{align*}
where the fact $\mathcal{N}(A_0) = \mathcal{N}(A_0^2)$ follows from the upcoming Lemma \ref{lem:ProjectionNullspace} and the considerations that follow in the proof of assumption (c). The surjectivity follows from the solvability of (\ref{eq:EllipticPDE1})-(\ref{eq:EllipticPDE2}) from above. Let $f \in \spn \{v_1, v_2\}^\perp \cap Z$. Then $\skp{f}{v_1}_{\tilde{L_2}} = \skp{f}{v_2}_{\tilde{L_2}} = 0$ and we know that there is a solution $u \in X_1$ with $\int_{\Gamma^*}\limits u^{(1)} \dH^2 = 0$ of $D_wF(\mathbb{O}) u = f$ and $P((D_wF)(\mathbb{O}) u) = P(f) = f$. Clearly this $u$ is in $\modulo{X_1}{\mathcal{N}(A_0)}$, since contributions of $v_0$, $v_1$ and $v_2$ do not affect $D_wF(\mathbb{O}) u = -A_0 u = f$. Moreover, $PF(\mathbb{O}) = P(0) = 0$ because $v^* \equiv 0$ corresponds to an SSC. By the same calculations as in Lemma 3.18 of \cite{Mue13} we see that $F$ and $F_w$ are continuous in a small neighborhood $U(\mathbb{O}) \subseteq X \times Y$ of $\mathbb{O}$ and so are $PF$ and $P F_w$. Therefore
\begin{align*}
	PF: U(\mathbb{O}) \subseteq X \times Y \longrightarrow \spn \{v_1, v_2\}^\perp \cap Z
\end{align*}
satisfies all assumptions of Theorem 4.B in \cite{Zei85}. So we see that there exist $r_0, r > 0$ such that for every $t \in \R^3$ with $\norm{t} \leq r_0$ there is exactly one $w(t) \in Y$ for which $\norm{w(t)}_Y \leq r$ and $PF(t,w(t)) = 0$. Hence
\begin{align*}
	\Psi: B_{r_0}(0) \subseteq \R^3 \longrightarrow \mathcal{E}: t = (t_0, t_1, t_2) \longmapsto \Psi(t) := t_0 v_0 + t_1 v_1 + t_2 v_2 + w(t)
\end{align*}
is the desired parametrization of $\hat{\mathcal{E}}$ in a neighborhood of $v^* \equiv 0$. Due to the fact that
\begin{align*}
	D\Psi(0) = \left(v_0 + (\d_{t_0} w)(0), v_1 + (\d_{t_1} w)(0), v_2 + (\d_{t_2} w)(0)\right)
\end{align*}
has full rank, because $v_0$, $v_1$ and $v_2$ are linearly independent and $w(t)$ belongs to $Y$, which is complementary to $\mathcal{N}(A_0)$, we see that $\hat{\mathcal{E}}$ is a $C^1$-manifold with $\dim(\hat{\mathcal{E}}) = 3$. \\
Next we try to find a $3$-dimensional manifold $\tilde{\mathcal{E}}$ that is contained in $\mathcal{E}$. We define
\begin{align*}
	\tilde{\mathcal{E}} := \left\{u \in V \cap X_1 \mid u \text{ parametrizes an SSC}\right\}.
\end{align*}
Then $\tilde{\mathcal{E}} \subseteq \mathcal{E}$ is obvious since for SSCs $F(u) = 0$ holds. For $|x|$, $|y|$, $|H - H^*|$ and $|R - R^*|$ small enough any $u \in \tilde{\mathcal{E}}$ is given implicitly as the solution of
\begin{align}\label{eq:ImplicitU}
	\norm{\Psi(q,u(q)) - \begin{pmatrix} x \\ y \\ H \end{pmatrix}}^2 = R^2 \qquad \forall \, q \in \Gamma^*,
\end{align}
where $\Psi$ is the curvilinear coordinate system as introduced in (\ref{eq:CurvilinearCoordinates}). And since this SC is also stationary, $u$ has to satisfy (\ref{eq:StationaryAngle}). The term $\cos(\alpha)$ can be replaced by $\frac{H}{R}$ and $r$ can be replaced by $r = \sqrt{R^2 - H^2}$ and so we obtain
\begin{align*}
	\frac{H}{R} = \frac{b}{\sqrt{R^2 - H^2}} - a,
\end{align*}
which is an equation that specifies the relation between $R$ and $H$. Therefore there is some way of expressing $H$ in terms of $R$ via a $C^1$-function $H(R)$ and this reduces the degrees of freedom in (\ref{eq:ImplicitU}) to three. It is again useful to write the curvilinear coordinate system in spherical coordinates. For $q = P(\phi,\theta)$ as in (\ref{eq:SphericalCoordinates}) we use the tangential correction terms $T(q)$ and $t(q,w)$ defined by
\begin{align*}
	T(q) := \begin{pmatrix}
				  \sin(\phi) \cos(\theta) \\
				  \cos(\phi) \cos(\theta) \\
				  -\sin(\theta)
			  \end{pmatrix} \qquad \text{ and } \qquad t(q,w) := -w \eta(\theta) \cot(\alpha^*),
\end{align*}
where $\eta: [0, \pi - \alpha^*] \longrightarrow \R: \theta \longmapsto \eta(\theta)$ is a smooth function that satisfies $|\eta(\theta)| \in [0,1]$ and $\eta(\pi - \alpha^*) = 1$. An easy computation shows that these choices guarantee $\Psi(q,w)|_{\d \Gamma^*} \in \d \Omega$ as required. Moreover, we see
\begin{align*}
	\d_w \Psi(q,0) = n_{\Gamma^*}(q) + t_w(q,0) T(q) = \begin{pmatrix}
																			\sin(\phi) \sin(\theta) \\
																			\cos(\phi) \sin(\theta) \\
																			\cos(\theta)
																		\end{pmatrix} + \eta(\theta) \cot(\alpha^*) \begin{pmatrix}
																																	  \sin(\phi) \cos(\theta) \\
																																	  \cos(\phi) \cos(\theta) \\
																																	  -\sin(\theta)
																																  \end{pmatrix}.
\end{align*}
Calculating the derivative of (\ref{eq:ImplicitU}) in $u \equiv 0$, which corresponds to the parameters $(0,0,R^*) \in \R^3$, we get
\begin{align*}
	&\left.2 \left(\Psi(q,u) - \begin{pmatrix} x \\ y \\ H(R) \end{pmatrix}\right) \cdot \d_w \Psi(q,u)\right|_{u \equiv 0}
						= 2 \left(\Psi(q,0) - \begin{pmatrix} 0 \\ 0 \\ H^* \end{pmatrix}\right) \cdot \d_w \Psi(q,0) \\
	&\hspace*{5mm} = 2 \left(q - \begin{pmatrix} 0 \\ 0 \\ H^* \end{pmatrix}\right)
			 \cdot \left(\begin{pmatrix}
								 \sin(\phi) \sin(\theta) \\
								 \cos(\phi) \sin(\theta) \\
								 \cos(\theta)
							 \end{pmatrix} + \eta(\theta) \cot(\alpha^*) \begin{pmatrix}
																							\sin(\phi) \cos(\theta) \\
																							\cos(\phi) \cos(\theta) \\
																							-\sin(\theta)
																						\end{pmatrix}\right) = 2R^* \neq 0.
\end{align*}
By the implicit function theorem and the fact that all the terms appearing in (\ref{eq:ImplicitU}) are smooth, there exists a three parameter family of $C^1$-functions $u(x,y,R)$ that parametrizes the SSCs in a neighbourhood of $\Gamma^*$. For $|x|$, $|y|$ and $|R - R^*|$ sufficiently small all these functions lie inside $\tilde{\mathcal{E}}$. Hence we found a parametrization
\begin{align*}
	\Phi: (-\epsilon_1, \epsilon_1) \times (-\epsilon_2, \epsilon_2) \times (R^* - \epsilon_3, R^* + \epsilon_3) \subseteq \R^3 \longrightarrow \tilde{\mathcal{E}}: (x,y,R) \longmapsto u(x,y,R)
\end{align*}
for sufficiently small $\epsilon_i > 0$ with $i \in \{1,2,3\}$, provided that $D\Psi(0,0,R^*)$ is not degenerated. The fact that $F(u(x,y,R)) = 0$ leads by differentiation to
\begin{align*}
	0 = D_uF(u(0,0,R^*)) u_x(0,0,R^*) = D_uF(0) u_x(0,0,R^*) = -A_0 u_x(0,0,R^*),
\end{align*}
which proves $u_x(0,0,R^*) \in \mathcal{N}(A_0)$. Similar we show $u_y(0,0,R^*), u_R(0,0,R^*) \in \mathcal{N}(A_0)$. This suggests that $u_x(0,0,R^*)$, $u_y(0,0,R^*)$ and $u_R(0,0,R^*)$ coincide with the functions $v_1$, $v_2$ and $v_0$ from (\ref{eq:NullspaceElements}). In fact, this can be calculated by differentiating
\begin{align*}
	\norm{\Psi(q,u(x,y,R)) - \begin{pmatrix} x \\ y \\ H(R) \end{pmatrix}}^2 - R^2 = 0
\end{align*}
with respect to $x$, $y$ and $R$ and evaluate it in $(0,0,R^*)$. We observe $u_x(0,0,R^*) = \sin(\phi) \sin(\theta)$, $u_y(0,0,R^*) = \cos(\phi) \sin(\theta)$ and $u_R(0,0,R^*) = H'(R^*) \cos(\theta) + R^*$. These functions are known to be linearly independent and therefore the rank of $D\Psi(0,0,R^*)$ is three. Hence $D\Psi(0,0,R^*)$ is non-degenerated and thus the proof of assumption (a) of the GLPS (see Theorem \ref{thm:GPLStability}) is complete.
\end{proof}

\begin{rem}
Actually we even proved a little more than assumption (a). We know by (\ref{eq:XShift})-(\ref{eq:ZShiftExp}) that there are three ways to transform the SSC $\Gamma^*$ into another SSC - namely an $x$-shift, a $y$-shift and a radial expansion with a simultaneous shift in $z$-direction. Knowing $\dim(\mathcal{E}) = 3$ we see that in a small neighborhood of $v^* \equiv 0$ the manifold of equilibria only consists of SSCs. And since we started with an arbitrary SSC $\Gamma^*$, we obtain the following result: ``Around an SSC the set $\mathcal{E}$ only consists of SSCs''. Unfortunately, this does not mean that SSCs are the only equilibria of (\ref{eq:Flow1})-(\ref{eq:Flow2}). Possibly there could be equilibria that are no SSCs, which are isolated or even form a manifold itself.
\end{rem}

Assumption (b) is an easy comparison of dimensions. We can see in (2.8) of \cite{PSZ09} that we always have $T_0\mathcal{E} \subseteq \mathcal{N}(A_0)$. This shows
\begin{align*}
	3 = \dim(\mathcal{E}) = \dim(T_0\mathcal{E}) \leq \dim(\mathcal{N}(A_0)) = 3,
\end{align*}
which leads to $T_0\mathcal{E} = \mathcal{N}(A_0)$ and thus proves assumption (b).

We continue with the proof of assumption (c). To this end the following two lemmas will be helpful.

\begin{lemma}\label{lem:ProjectionNullspace}
Let $P: X_0 \longrightarrow \mathcal{R}(P) = \mathcal{N}(A_0)$ be a projection and $P A_0 = A_0 P (= 0)$, then $\mathcal{N}(A_0) = \mathcal{N}(A_0^2)$. 
\end{lemma}
\begin{proof}
The inclusion $\mathcal{N}(A_0) \subseteq \mathcal{N}(A_0^2)$ is trivial. Hence assume $v \in \mathcal{N}(A_0^2)$, then $A_0^2 v = 0$, which means $A_0 v \in \mathcal{N}(A_0)$. $P$ applied to an element of $\mathcal{N}(A_0)$ is the identity and we obtain $A_0 v = P A_0 v = A_0 P v = 0$, which shows $v \in \mathcal{N}(A_0)$.
\end{proof}

\begin{lemma}\label{lem:ProjectionSemiSimple}
Assume $\mathcal{N}(A_0) = \mathcal{N}(A_0^2)$. Then $X_0 = \mathcal{N}(A_0) \oplus \mathcal{R}(A_0)$, which means $0$ is a semi-simple eigenvalue of $A_0$.
\end{lemma}
\begin{proof}
Let $\mu \in (0,\infty)$ satisfy $\mu \notin \sigma(-A_0)$ and define
\begin{align*}
	B := (\mu \id + A_0)^{-1}: X_0 \longrightarrow X_0.
\end{align*}
Simple algebraic manipulations show that
\begin{itemize}
\item[$\bullet$] $\frac{1}{\mu}$ is an eigenvalue of $B$, 
\item[$\bullet$] $\mathcal{N}\left(\frac{1}{\mu} \id - B\right) = \mathcal{N}(A_0)$,
\item[$\bullet$] $\mathcal{N}\left(\frac{1}{\mu} \id - B\right) = \mathcal{N}\left(\left(\frac{1}{\mu} \id - B\right)^2\right)$,
\item[$\bullet$] $\mathcal{R}\left(\frac{1}{\mu} \id - B\right) = \mathcal{R}(A_0)$.
\end{itemize}
Due to the compact embedding of $X_1 \hookrightarrow X_0$ the operator
\begin{align*}
	B = (\mu \id + A_0)^{-1}: X_0 \longrightarrow X_1 \hookrightarrow X_0
\end{align*}
is compact as a composition of a bounded and a compact operator. The spectral theorem for compact operators shows
\begin{align*}
	X_0 = \mathcal{N}\left(\frac{1}{\mu} \id - B\right) \oplus \mathcal{R}\left(\frac{1}{\mu} \id - B\right)
\end{align*}
and due to the above identities we get $X_0 = \mathcal{N}(A_0) \oplus \mathcal{R}(A_0)$.
\end{proof}

\begin{rem}\label{rem:MuInRhoA0}
In Theorem \ref{thm:Semigroup2} we saw that $-A_0$ is the generator of an analytic semigroup, which means that this operator is sectorial. Hence there exists $\omega \in \R$ and $\vartheta \in (\frac{\pi}{2}, \pi)$ such that $\rho(-A_0) \supseteq S_{\omega,\vartheta} := \{z \in \C \mid z \neq \omega, |\arg(z - \omega)| < \vartheta\}$. Especially $\rho(-A_0)$ contains the interval $(\omega,\infty)$ and one can always find $\mu \in (0,\infty)$, which satisfies $\mu \notin \sigma(-A_0)$ as required in the proof of Lemma \ref{lem:ProjectionSemiSimple}. Also by the sectoriality of $-A_0$ we know that $\norm{(\mu \id + A_0)^{-1}}_{\mathcal{L}(X_0)} \leq \frac{M}{|\mu - \omega|}$ for all $\mu \in S_{\omega,\vartheta}$, which justifies the boundedness of $B$ in the 5th step of the previous proof.
\end{rem}

By Lemma \ref{lem:ProjectionNullspace} and \ref{lem:ProjectionSemiSimple} we see that it is enough for $0$ to be a semi-simple eigenvalue to find a projection as in the assumptions of Lemma \ref{lem:ProjectionNullspace}. Indeed we can find such a projection, which is given by
\begin{align*}
	P: X_0 \longrightarrow \mathcal{N}(A_0): v = (v^{(1)}, v^{(2)}) \longmapsto P(v) := a_0(v) v_0 + a_1(v) v_1 + a_2(v) v_2,
\end{align*}
where the coefficients $a_i$ are defined as follows
\begin{align*}
	a_0(v) & := \frac{\int_{\Gamma^*}\limits v^{(1)} \dH^2}{\int_{\Gamma^*}\limits v^{(1)}_0 \dH^2}, \\
	a_1(v) & := \frac{\skp{v^{(1)}}{v^{(1)}_1}_{L_2(\Gamma^*)} + \frac{1}{\sin(\alpha^*)^2} \int_{\d \Gamma^*}\limits v^{(2)} v^{(2)}_1 \dH^1}{\skp{v^{(1)}_1}{v^{(1)}_1}_{L_2(\Gamma^*)} + \frac{1}{\sin(\alpha^*)^2} \int_{\d \Gamma^*}\limits v^{(2)}_1 v^{(2)}_1 \dH^1}, \\
	a_2(v) & := \frac{\skp{v^{(1)}}{v^{(1)}_2}_{L_2(\Gamma^*)} + \frac{1}{\sin(\alpha^*)^2} \int_{\d \Gamma^*}\limits v^{(2)} v^{(2)}_2 \dH^1}{\skp{v^{(1)}_2}{v^{(1)}_2}_{L_2(\Gamma^*)} + \frac{1}{\sin(\alpha^*)^2} \int_{\d \Gamma^*}\limits v^{(2)}_2 v^{(2)}_2 \dH^1}
\end{align*}
with $v_0$, $v_1$ and $v_2$ as the elements from (\ref{eq:NullspaceElements}) spanning the nullspace. This projection has the desired properties, because obviously $\mathcal{R}(P) = \mathcal{N}(A_0)$ since $v_0$, $v_1$ and $v_2$ span the nullspace of $A_0$. Moreover, $P|_{\mathcal{N}(A_0)} = \id_{\mathcal{N}(A_0)}$ or equivalently $P^2 = P$, because $a_i(v_j) = \delta_{ij}$ for $i,j \in \{0, 1, 2\}$ as one can see by elementary but time-consuming calculations (cf. pages 133ff. of \cite{Mue13}). Furthermore, $P A_0 = 0$ as one can see by
\begin{align*}
	\int_{\Gamma^*} A_0 v^{(1)} \dH^2 & = \int_{\Gamma^*} \left(\Delta^B v^{(1)} - \mint_{\Gamma^*} \Delta^B v^{(1)} \dH^2\right) \dH^2 \\
		& = \int_{\Gamma^*} \Delta^B v^{(1)} \dH^2 - \left(\mint_{\Gamma^*} \Delta^B v^{(1)} \dH^2\right) \int_{\Gamma^*} 1 \dH^2 = 0,
\end{align*}
and $\skp{(A_0 v)^{(1)}}{v^{(1)}_i}_{L_2(\Gamma^*)} = \frac{-1}{\sin(\alpha^*)^2} \int_{\d \Gamma^*} (A_0 v)^{(2)} v_i^{(2)} \dH^1$ for $i \in \{1, 2\}$. Hence $P A_0 = 0 (= A_0 P)$ and having found this projection we completed the prove of assumption (c).

The last assumption we have to check for Theorem \ref{thm:GPLStability} is (d). Here we will see that the eigenvalues of $A_0$ can be traced back to the eigenvalues of the operator $\Delta^B$. Since we want to show that $\sigma(A_0) \setminus \{0\}$ is contained in the complex right half-plane, we can ignore eigenfunctions corresponding to the eigenvalue $0$. Assume that $(\lambda,u)$ is an eigenpair of $A_0$ with $\lambda \neq -\frac{2}{{R^*}^2} = -|\sigma^*|^2$. Then we first remark that it is not possible for $u^{(1)}$ to be constant, since otherwise
\begin{align*}
	(A_0 u)^{(1)} = \underbrace{\Delta^B u^{(1)}}_{= 0} - \mint_{\Gamma^*} \underbrace{\Delta^B u^{(1)}}_{= 0} \dH^2 = 0
\end{align*}
and $u$ would correspond to the eigenvalue $0$, which is not considered. \\
Due to $\lambda \neq -|\sigma^*|^2$ the constant
\begin{align*}
	c(\lambda,u) := \begin{pmatrix}
							 \dfrac{1}{\lambda + |\sigma^*|^2} \mint_{\Gamma^*}\limits \Delta^B u^{(1)} \dH^2 \\
							 0
						 \end{pmatrix}
\end{align*}
is well-defined and the function $\tilde{u} := u + c(\lambda,u) \not\equiv 0$ is an eigenfunction of $\Delta^B$, as one can see from
\begin{align*}
	\Delta^B \tilde{u}^{(1)} & = \Delta^B u^{(1)} - \underbrace{\Delta_{\Gamma^*} c(\lambda,u)}_{= 0} - |\sigma^*|^2 c(\lambda,u) \\
									 & = \Delta^B u^{(1)} - \frac{|\sigma^*|^2}{\lambda + |\sigma^*|^2} \mint_{\Gamma^*} \Delta^B u^{(1)} \dH^2 \\
									 & = \Delta^B u^{(1)} - \mint_{\Gamma^*} \Delta^B u^{(1)} \dH^2
										+ \frac{\lambda}{\lambda + |\sigma^*|^2} \mint_{\Gamma^*} \Delta^B u^{(1)} \dH^2 \\
									 & = A_0 u^{(1)} + \lambda c(\lambda,u) = \lambda u^{(1)} + \lambda c(\lambda,u) = \lambda \tilde{u}^{(1)}.
\end{align*}
Obviously, the second component of $\Delta^B \tilde{u}$ does not change compared to $\Delta^B u$. This argument does not work for $\lambda = -\frac{2}{{R^*}^2}$. Therefore we have shown
\begin{align}\label{eq:EVsOfA0AndDeltaB}
	\sigma(A_0) \subseteq \sigma(\Delta^B) \cup \left\{-\frac{2}{{R^*}^2}\right\}.
\end{align}
Remember that we have already proven some statements concerning the eigenvalues of $\Delta^B$. For example we have seen that all eigenvalues of $\Delta^B$ are real. Since also $-\frac{2}{{R^*}^2}$ is in $\R$, all eigenvalues of $A_0$ are real. With this knowledge the proof of assumption (d) relies on the following argument: \\
If one real eigenvalue of $A_0$ would change its sign while varying the parameters $(a,b)$, it would also become $0$ at some point, provided that the eigenvalues depend continuously on $(a,b)$. But this would cause $\mathcal{N}(A_0)$ to be higher-dimensional than before. We have already seen that independent of the choice of $a > -1$ and $b > C_{crit}$ the nullspace $\mathcal{N}(A_0)$ is always $3$-dimensional. For this reason $\sigma(A_0) \setminus \{0\} \subseteq \R_+ \subseteq \C_+$ has to hold as long as the varied parameters do not violate the condition $a > -1$ and $b > C_{crit}$.

So the strategy to prove (d) will be as follows:
\begin{enumerate}
	\item Show that the eigenvalues of $A_0$ depend continuously on the parameters $a$ and $b$.
	\item Find a particular parameter setting $(a_0, b_0)$, where we can easily show that the spectrum of $A_0$ is contained in $[0,\infty)$.
	\item Starting from the particular setting $(a_0,b_0)$, vary the parameters to cover a wider parameter range.
\end{enumerate}

We start by showing the continuous dependence of the eigenvalues on $(a,b)$. Obviously, $\cos(\alpha^*)$, $\sin(\alpha^*)$ and $R^*$ depend continuously on the parameters $a > -1$ and $b > 0$ and so do all coefficients appearing in $A_0$ and hence also $A_0$ itself. Therefore we can show
\begin{align*}
	A_0(\tilde{a},\tilde{b}) \xrightarrow[(\tilde{a},\tilde{b}) \rightarrow (a,b)]{} A_0(a,b) \qquad \text{ in } \mathcal{L}(\mathcal{D}(A_0),X_0),
\end{align*}
where $\mathcal{D}(A_0)$ is equipped with the graph norm. Lemma A.3.1 from \cite{Lun95} shows that
\begin{align*}
	(\lambda \id - A(\tilde{a},\tilde{b}))^{-1} \xrightarrow[(\tilde{a},\tilde{b}) \rightarrow (a,b)]{} (\lambda \id - A(a,b))^{-1} \qquad \text{ in } \mathcal{L}(X_0).
\end{align*}
Using Theorem 2.25 of \cite{Kat95} we see that $A_0(\tilde{a},\tilde{b}) \xrightarrow[(\tilde{a},\tilde{b}) \rightarrow (a,b)]{} A_0(a,b)$ in the generalized sense (cf. IV-§ 2 in \cite{Kat95}). In doing so it is important to remark that $A_0$ is closed, because the resolvent set is not empty. Section IV-§ 3.5 of \cite{Kat95} shows that each finite system of eigenvalues depends continuously on $(a,b)$. We saw in Remark \ref{rem:IsolatedEVs} that all eigenvalues of $\Delta^B$ are isolated and one possible new eigenvalue does not change this fact for $A_0$. After every eigenvalue of $A_0$ is isolated, the one-element set $\{\lambda_i\}$ forms such a finite system and therefore depends continuously on the parameters $(a,b)$. This completes the first part of our strategy towards assumption (d).

Now we search for a situation, where we can easily compute the eigenvalues of $A_0$. We find this in the halfsphere. We choose an arbitrary $a_0 > 0$. By (\ref{eq:RangeOfAlpha}) we know that for this choice of $a_0$ an angle $\cos(\alpha^*) = 0$ is always possible. For the moment the parameter $b_0 > 0$ could be chosen arbitrarily since $\cos(\alpha^*) = 0$ simplifies $b_0 > C_{crit}$ to $b_0 > 0$. But for later purpose we choose $b_0 \in (0,1)$. We set $r^* = \frac{b_0}{a_0}$ and obtain a stationary halfsphere. The reason why we choose $\Gamma^*$ to be the halfsphere is that by its reflection along the $x$-$y$-plane, called $-\Gamma^*$, the resulting surface $\Gamma^* \cup -\Gamma^*$ is smooth. \\
Due to (\ref{eq:EVsOfA0AndDeltaB}) the eigenvalue problem we have to solve is
\begin{align}\label{eq:SpecialEP}
	\lambda \rho = \Delta^B \rho = \begin{pmatrix}
												 -\dfrac{1}{{R^*}^2 \sin(\theta)^2} \rho^{(1)}_{\phi \phi}
												 - \dfrac{1}{{R^*}^2} \rho^{(1)}_{\theta \theta}
												 - \dfrac{1}{{R^*}^2} \cot(\theta) \rho^{(1)}_\theta - \dfrac{2}{{R^*}^2} \rho^{(1)} \\
												 \dfrac{1}{R^*} \rho^{(1)}_\theta(\pi - \alpha^*)
												 - \dfrac{b_0}{{R^*}^2} (\rho^{(2)}_{\phi \phi} + \rho^{(2)})
											 \end{pmatrix},
\end{align}
where we have to impose $2\pi$-periodicity in $\phi$ and continuity for $\theta = 0$. To avoid unnecessary terms we multiply by ${R^*}^2$, add $2 \rho$ and obtain
\begin{align*}
	({R^*}^2\lambda + 2) \rho = \begin{pmatrix}
											 -\dfrac{1}{\sin(\theta)^2} \rho^{(1)}_{\phi \phi} - \rho^{(1)}_{\theta \theta}
											 - \cot(\theta) \rho^{(1)}_\theta \\
											 R^* \rho^{(1)}_\theta(\pi - \alpha^*)
											 - b_0 (\rho^{(2)}_{\phi \phi} + \rho^{(2)}) + 2 \rho^{(2)}
										 \end{pmatrix}.
\end{align*}
Then we substitute $\mu := {R^*}^2 \lambda + 2$ and search for all values $\mu$ can attain. Having a reflectional symmetric $\Gamma^*$ is important but not enough. We also need smoothly reflectable eigenfunctions, i.e. eigenfunctions with $\rho_\theta|_{\theta = \pi - \alpha^*} = 0$. To achieve this we have to introduce one more parameter $d \in [0,1]$ and solve
\begin{align}\label{eq:DeltaD}
	\begin{pmatrix}
		\mu \rho^{(1)} \\
		d \mu \rho^{(2)}
	\end{pmatrix} = \begin{pmatrix}
							 -\dfrac{1}{\sin(\theta)^2} \rho^{(1)}_{\phi \phi} - \rho^{(1)}_{\theta \theta}
							 - \cot(\theta) \rho^{(1)}_\theta \\
							 R^* \rho^{(1)}_\theta(\pi - \alpha^*) - d b_0 (\rho^{(2)}_{\phi \phi} + \rho^{(2)}) + 2d \rho^{(2)}
						 \end{pmatrix} =: \Delta^d \rho
\end{align}
on the halfsphere $\Gamma^*$. For $d = 0$ this reads as
\begin{align*}
	\mu \rho^{(1)} = -\frac{1}{\sin(\theta)^2} \rho^{(1)}_{\phi \phi} - \rho^{(1)}_{\theta \theta} - \cot(\theta) \rho^{(1)}_\theta
\end{align*}
with the boundary condition $\rho^{(1)}_\theta(\pi - \alpha^*) = 0$. Together with the $2\pi$-periodicity in $\phi$ and the continuity for $\theta = 0$ we see that any solution of this problem on the halfsphere $\Gamma^*$ can be smoothly reflected to a solution of
\begin{align*}
	\mu \rho^{(1)} = -\frac{1}{\sin(\theta)^2} \rho^{(1)}_{\phi \phi} - \rho^{(1)}_{\theta \theta} - \cot(\theta) \rho^{(1)}_\theta
\end{align*}
on the full sphere $\Gamma^* \cup -\Gamma^*$, with periodicity in $\phi$ and continuity for $\theta = 0$ and $\theta = \pi$. Yet, this eigenvalue problem for the Laplace operator on the sphere is already well studied by different authors - for example by \cite{CH68}, \cite{Tri72} or chapter XIII in \cite{Jae01}. As each of these sources shows, the eigenvalues of this equation are given as $k (k + 1)$ for $k \in \N$. Thus $\mu_k = k (k + 1)$ and for $\lambda_k$ we have the equation $({R^*}^2 \lambda_k + 2) = k (k + 1)$, which leads to
\begin{align}\label{eq:SpecialEVs}
	\lambda_k = \frac{k (k + 1) - 2}{{R^*}^2} \qquad \text{ for every } k \in \N.
\end{align}
Obviously, we see $\lambda_k \geq 0$ for $k \geq 1$ and the only eigenvalue that could cause a problem is $\lambda_0 = -\frac{2}{{R^*}^2}$. We will see later that although $\lambda_0 = -\frac{2}{{R^*}^2}$ is a possible eigenvalue of $\Delta^B$ it is not possible as eigenvalue for $A_0$. \\
Now we want to increase the parameter $d$ from $0$ to $1$. We will need the continuous dependence of the eigenvalues on $d$ to argue that while increasing $d$ no eigenvalue can change its sign. This is again due to the fact that the nullspace is three dimensional. Although we have not included the weight $d$ into the considerations concerning the nullspace previously in this section, the calculations do not change dramatically and we also get that the nullspace is always $3$-dimensional for all $d \in [0,1]$. Therefore the continuous dependence of the eigenvalues on $d$ is the next ingredient that we are going to prove. \\
With
\begin{align*}
	\Delta^{-d}: H(d) := L_2(\Gamma^*) \times L_2(\d \Gamma^*) \longrightarrow H(d)
\end{align*}
we denote the inverse operator of $\Delta^d + c \id$, where $H(d)$ shall be equipped with the inner product $\skp{u}{v}_{H(d)} := \skp{u^{(1)}}{v^{(1)}}_{L_2(\Gamma^*)} + d \skp{u^{(2)}}{v^{(2)}}_{L_2(\d \Gamma^*)}$. Moreover, we assume that $c$ is large enough to guarantee that all eigenvalues are positive. Since we only want to show the continuous dependence of the eigenvalues, we do not care for shifts of the operator and the resulting shift of the spectrum. We consider the inverse operator since its spectrum is bounded, which will be important later on. Assuming that we have a solution $\rho$ of the equation (\ref{eq:DeltaD}) we get
\begin{align}\label{eq:TestedEF}
	& \mu \skp{\rho^{(1)}}{\rho^{(1)}}_{L_2(\Gamma^*)} + d \mu \skp{\rho^{(2)}}{\rho^{(2)}}_{L_2(\d \Gamma^*)} \notag \\
	& \hspace*{5mm} = \skp{(\Delta^d \rho)^{(1)}}{\rho^{(1)}}_{L_2(\Gamma^*)} + \skp{(\Delta^d \rho)^{(2)}}{\rho^{(2)}}_{L_2(\d \Gamma^*)} \notag \\
	& \hspace*{5mm} = \int_{\Gamma^*} -(\Delta_{\Gamma^*} \rho^{(1)}) \rho^{(1)} \dH^2 + \int_{\d \Gamma^*} \left(-(n_{\d \Gamma^*} \cdot \nabla_{\Gamma^*} \rho^{(1)}) + d b_0 \rho^{(2)}_{\sigma\sigma} + d \frac{b_0}{{R^*}^2} \rho^{(2)}\right) \rho^{(2)} \dH^1 \notag \\
	& \hspace*{5mm} = \int_{\Gamma^*} \norm{\nabla_{\Gamma^*} \rho^{(1)}}^2 \dH^2 - \int_{\d \Gamma^*} d b_0 (\rho^{(2)}_\sigma)^2 + d \frac{b_0}{{R^*}^2} (\rho^{(2)})^2 \dH^1.
\end{align}
If we denote the eigenvalues of $\Delta^{-d}$ by $\nu$, this can be rewritten as
\begin{align}\label{eq:ExpressionForNu}
	\nu = \frac{\skp{\Delta^{-d} \rho}{\rho}_{H(1)}}{\skp{\rho}{\rho}_{H(d)}}.
\end{align}
This representation is all we need for Courant's maximum-minimum principle (cf. Chap. VII §1.4 in \cite{CH68}) to see that for a fixed $d$ the eigenvalues $\nu_k(d)$ can be written as
\begin{align*}
	\nu_k(d) = \max_{W \in \Sigma_k} \min_{\rho \in W \setminus \{0\}} \frac{\skp{\Delta^{-d} \rho}{\rho}_{H(1)}}{\skp{\rho}{\rho}_{H(d)}},
\end{align*}
where $\Sigma_k$ denotes the set of all $k$-dimensional subspaces of $H(d)$. Now we want to sketch the continuous dependence of $\nu_k(d)$ on $d$. With $E_k(d) = \spn \{\rho_1(d), \ldots, \rho_k(d)\}$ as the span of the first $k$ eigenfunctions, we estimate
\begin{align*}
	\nu_k(d_1) - \nu_k(d_2) & \geq \min_{\rho \in E_k(d_2) \setminus \{0\}} \frac{\skp{\Delta^{-d_1} \rho}{\rho}_{H(1)}}{\skp{\rho}{\rho}_{H(d_1)}} - \min_{\rho \in E_k(d_2) \setminus \{0\}} \frac{\skp{\Delta^{-d_2} \rho}{\rho}_{H(1)}}{\skp{\rho}{\rho}_{H(d_2)}},
\end{align*}
since the second maximum is attained exactly for $E_k(d_2)$ and the first summand gets smaller if we consider this particular choice. Then we are able to choose $\hat{\rho} \in E_k(d_2)$ with
\begin{align}\label{eq:NormedEF}
	\skp{\hat{\rho}}{\hat{\rho}}_{H(d_1)} = 1
\end{align}
such that the first minimum is attained and get
\begin{align*}
	\nu_k(d_1) - \nu_k(d_2) \geq \skp{\Delta^{-d_1} \hat{\rho}}{\hat{\rho}}_{H(1)} - \frac{\skp{\Delta^{-d_2} \hat{\rho}}{\hat{\rho}}_{H(1)}}{\skp{\hat{\rho}}{\hat{\rho}}_{H(d_2)}}.
\end{align*}
This can be rewritten to
\begin{align*}
	\nu_k(d_1) - \nu_k(d_2) & \geq \skp{\Delta^{-d_1} \hat{\rho}}{\hat{\rho}}_{H(1)} - \frac{\skp{\Delta^{-d_2} \hat{\rho}}{\hat{\rho}}_{H(1)}}{\skp{\hat{\rho}}{\hat{\rho}}_{H(d_2)}} - \skp{\Delta^{-d_2} \hat{\rho}}{\hat{\rho}}_{H(1)} + \skp{\Delta^{-d_2} \hat{\rho}}{\hat{\rho}}_{H(1)} \\
									& = \skp{(\Delta^{-d_1} - \Delta^{-d_2}) \hat{\rho}}{\hat{\rho}}_{H(1)} + \left(1 - \frac{1}{\skp{\hat{\rho}}{\hat{\rho}}_{H(d_2)}}\right) \skp{\Delta^{-d_2} \hat{\rho}}{\hat{\rho}}_{H(1)}.
\end{align*}
The appearing denominator can be written as
\begin{align*}
	\skp{\hat{\rho}}{\hat{\rho}}_{H(d_2)} & = \skp{\hat{\rho}^{(1)}}{\hat{\rho}^{(1)}}_{L_2(\Gamma^*)} + d_2 \skp{\hat{\rho}^{(2)}}{\hat{\rho}^{(2)}}_{L_2(\d \Gamma^*)} \\
		& + d_1 \skp{\hat{\rho}^{(2)}}{\hat{\rho}^{(2)}}_{L_2(\d \Gamma^*)} - d_1 \skp{\hat{\rho}^{(2)}}{\hat{\rho}^{(2)}}_{L_2(\d \Gamma^*)} \\
		& = 1 + (d_2 - d_1) \skp{\hat{\rho}^{(2)}}{\hat{\rho}^{(2)}}_{L_2(\d \Gamma^*)}
\end{align*}
and hence we end up with
\begin{align*}
	\nu_k(d_1) - \nu_k(d_2) & \geq \skp{(\Delta^{-d_1} - \Delta^{-d_2}) \hat{\rho}}{\hat{\rho}}_{H(1)} \\
		& + \left(1 - \frac{1}{1 + (d_2 - d_1) \skp{\hat{\rho}^{(2)}}{\hat{\rho}^{(2)}}_{L_2(\d \Gamma^*)}}\right) \skp{\Delta^{-d_2} \hat{\rho}}{\hat{\rho}}_{H(1)}.
\end{align*}
If we consider the limit $d_2 \longrightarrow d_1$, we first of all observe that the first term on the right-hand side converges to zero which can be see similar as in Subsection 2.3.1 of \cite{Hen06}. It might be noteworthy that the proofs of Theorem 2.3.1 and Theorem 2.3.2 in \cite{Hen06} contain two little mistakes: In the proof of Theorem 2.3.1 the minimum and maximum must be interchanged and in the proof of Theorem 2.3.2 equation (2.22) should estimate the norm $\norm{A_n}_{\mathcal{L}(H^{-1},H^1_0)}$ as this is used in the last line of the proof, but instead it estimates $\norm{A_n}_{\mathcal{L}(L_2,L_2)}$. But the argument previous to (2.22) also justifies this modification and the result remains unchanged. Then we immediately see that
\begin{align*}
	\lim_{d_2 \rightarrow d_1} \inf \left(\nu_k(d_1) - \nu_k(d_2)\right) \geq 0
\end{align*}
as long as $\skp{\hat{\rho}^{(2)}}{\hat{\rho}^{(2)}}_{L_2(\d \Gamma^*)}$ remains bounded independent of $d$. In fact, for an eigenfunction $\rho$ that satisfies (\ref{eq:NormedEF}), the equation (\ref{eq:ExpressionForNu}) shows that
\begin{align*}
	\skp{\Delta^{-d_1} \rho}{\rho}_{H(1)} = \nu \skp{\rho}{\rho}_{H(d_1)} = \nu \leq c < \infty,
\end{align*}
because the eigenvalues of $\Delta^{-d_1}$ are bounded. Yet, controlling $\skp{\Delta^{-d_1} \rho}{\rho}_{H(1)}$ is due to (\ref{eq:TestedEF}) equivalent to controlling the $H^1$-norm of $\rho^{(1)}$, given by
\begin{align*}
	\int_{\Gamma^*} \norm{\nabla_{\Gamma^*} \rho^{(1)}}^2 \dH^2
\end{align*}
for all $d \in [0,1]$. Since
\begin{align*}
	W^1_2(\Gamma^*) \xrightarrow[]{\gamma_0} W^{\frac{1}{2}}_2(\d \Gamma^*) \hookrightarrow L_2(\d \Gamma^*)
\end{align*}
this also controls the $L_2(\d \Gamma^*)$-norm of $\gamma_0 \rho^{(1)} = \rho^{(2)}$, which is what we need. Interchanging the roles of $d_1$ and $d_2$, we also get the converse inequality
\begin{align*}
	\lim_{d_2 \rightarrow d_1} \inf \left(\nu_k(d_2) - \nu_k(d_1)\right) \geq 0.
\end{align*}
Thus $\lim_{d_2 \rightarrow d_1}\limits \nu_k(d_2) = \nu_k(d_1)$ and we obtain the continuous dependence of $\nu_k$ and therefore also of $\mu_k$ on $d$.

We know that for $d = 0$ all but two eigenvalues are positive and independent of $d$ and the nullspace is always $3$-dimensional. If we now increase $d$ from $0$ to $1$, which leads to $\Delta^B$, no eigenvalue can change its sign. Hence all eigenvalues of $\Delta^B$ except $0$ and $\lambda_0$ are positive in this halfsphere case. \\
We still have to exclude $\lambda_0$ for $A_0$. If we assume $\lambda_0 = -\frac{2}{{R^*}^2} \in \sigma(A_0)$ and $\rho_0$ to be an eigenfunction corresponding to $\lambda_0$, we obtain
\begin{align*}
	& & (A_0 \rho_0)^{(1)} &= -\frac{2}{{R^*}^2} \rho_0^{(1)} \\
	&\Rightarrow \quad & -\Delta_{\Gamma^*} \rho_0^{(1)} - \frac{2}{{R^*}^2} \rho_0^{(1)} - \mint_{\Gamma^*} \Delta^B \rho_0^{(1)} \dH^2 &= -\frac{2}{{R^*}^2} \rho_0^{(1)} \\
	&\Rightarrow \quad & \Delta_{\Gamma^*} \rho_0^{(1)} &= -\mint_{\Gamma^*} \Delta^B \rho_0^{(1)} \dH^2 = const. \\
	&\Rightarrow \quad & \Delta_{\Gamma^*} \rho_0^{(1)} &= \mint_{\Gamma^*} \Delta_{\Gamma^*} \rho_0^{(1)} + \frac{2}{{R^*}^2} \rho_0^{(1)} \dH^2 \\
	&\Rightarrow \quad & \Delta_{\Gamma^*} \rho_0^{(1)} &= \Delta_{\Gamma^*} \rho_0^{(1)} \underbrace{\mint_{\Gamma^*} 1 \dH^2}_{= 1} + \frac{2}{{R^*}^2} \mint_{\Gamma^*} \rho_0^{(1)} \dH^2 \\
	&\Rightarrow \quad & \int_{\Gamma^*} \rho_0^{(1)} \dH^2 &= 0.
\end{align*}
This shows that an eigenfunction $\rho_0$ would satisfy $\Delta_{\Gamma^*} \rho_0^{(1)} = c$ and $\int_{\Gamma^*}\limits \rho_0^{(1)} \dH^2 = 0$. This can be used to calculate
\begin{align*}
	0 & = c \int_{\Gamma^*} \rho_0^{(1)} \dH^2 = \int_{\Gamma^*} \rho_0^{(1)} \Delta_{\Gamma^*} \rho_0^{(1)} \dH^2 \\
	  & = -\int_{\Gamma^*} \nabla_{\Gamma^*} \rho_0^{(1)} \cdot \nabla_{\Gamma^*} \rho_0^{(1)} \dH^2 + \int_{\d \Gamma^*} (n_{\d \Gamma^*} \cdot \nabla_{\Gamma^*} \rho_0^{(1)}) \rho_0^{(1)} \dH^2,
\end{align*}
which can be written as
\begin{align*}
	\norm{\nabla_{\Gamma^*} \rho_0^{(1)}}_{L_2(\Gamma^*)}^2 = \int_{\d \Gamma^*} (n_{\d \Gamma^*} \cdot \nabla_{\Gamma^*} \rho_0^{(1)}) \rho_0^{(1)} \dH^2.
\end{align*}
Utilizing the so far unused second component of $A_0 \rho_0$ we get
\begin{align*}
	\frac{2}{{R^*}^2} \rho_0^{(2)} = -(\lambda_0 \rho_0)^{(2)} = -(A_0 \rho_0)^{(2)} = -(n_{\d \Gamma^*} \cdot \nabla_{\Gamma^*} \rho_0^{(1)}) + b_0 (\rho_0^{(2)})_{\sigma \sigma} + \frac{b_0}{{R^*}^2} \rho_0^{(2)}
\end{align*}
or equivalently
\begin{align*}
	(n_{\d \Gamma^*} \cdot \nabla_{\Gamma^*} \rho_0^{(1)}) = -\frac{2}{{R^*}^2} \rho_0^{(2)} + b_0 (\rho_0^{(2)})_{\sigma \sigma} + \frac{b_0}{{R^*}^2} \rho_0^{(2)} = \frac{b_0 - 2}{{R^*}^2} \rho_0^{(2)} + b_0 (\rho_0^{(2)})_{\sigma\sigma}.
\end{align*}
This can be used to transform the calculation before into
\begin{align*}
	\norm{\nabla_{\Gamma^*} \rho_0^{(1)}}_{L_2(\Gamma^*)}^2
		& = \int_{\d \Gamma^*} \frac{b_0 - 2}{{R^*}^2} (\rho_0^{(2)})^2 \dH^1 + \int_{\d \Gamma^*} b_0 (\rho_0^{(2)})_{\sigma\sigma} \rho_0^{(2)} \dH^1 \\
		& = \frac{b_0 - 2}{{R^*}^2} \norm{\rho_0^{(2)}}_{L_2(\d \Gamma^*)} - b_0 \int_{\d \Gamma^*} (\rho_0^{(2)})_{\sigma}^2 \dH^1
\end{align*}
and finally end up with
\begin{align}\label{eq:NormContradiction}
	\norm{\nabla_{\Gamma^*} \rho_0^{(1)}}_{L_2(\Gamma^*)}^2 + b_0 \norm{(\rho_0^{(2)})_{\sigma}}_{L_2(\d \Gamma^*)}^2 = \frac{b_0 - 2}{{R^*}^2} \norm{\rho_0^{(2)}}_{L_2(\d \Gamma^*)}.
\end{align}
Here we reached the point where the choice $b_0 \in (0,1)$ is paying off. Since the numerator is negative, the right-hand side itself is negative. This leads again to a contradiction and shows that $\lambda_0 = -\frac{2}{{R^*}^2}$ is not an eigenvalue of $A_0$. Thus we found the ``easy'' situation, where every non-zero eigenvalue of $A_0$ is positive and can come to the last step for proving assumption (d).

Now we can vary the parameters starting from $(a_0,b_0)$ to cover a wide range, where the eigenvalues are positive. We start by noting that all the coefficients appearing in $A_0$ will not degenerate, because $R^* \neq 0$ and $\sin(\alpha^*) \neq 0$. As we said before the only important restriction comes from the $3$-dimensionality of the nullspace $\mathcal{N}(A_0)$. We saw that we can guarantee this as long as
\begin{align*}
	b > C_{crit} = -\frac{1}{3} R^* \sin(\alpha^*)^2 \cos(\alpha^*) = -\frac{r^*}{3} \sqrt{1 - \left(\frac{b}{r^*} - a\right)^2} \left(\frac{b}{r^*} - a\right).
\end{align*}
This varying process will require several steps and Figure \ref{fig:CriticalCase} is visualizing the upcoming situation.

\begin{figure}[htbp]
	\centering
	\includegraphics[width=140mm]{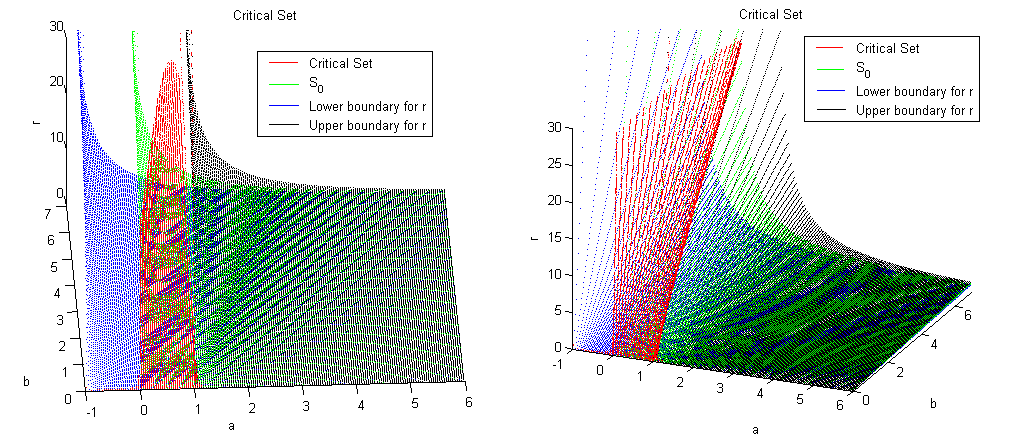}
	\caption{Critical parameter set}
	\label{fig:CriticalCase}
\end{figure}

First we consider the set corresponding to SSCs with $\cos(\alpha^*) = 0$ given by
\begin{align*}
	S_0 := \left\{(a,b,r^*) \in \R^3 \left| \, a > 0, b > 0, \frac{b}{r^*} = a\right.\right\}
		 = \left\{\left.\left(a,b,\frac{b}{a}\right) \in \R^3 \right| \, a > 0, b > 0\right\}.
\end{align*}
Let $(a_1,b_1) \in (0,\infty) \times (0,\infty)$ be arbitrary and consider the variation
\begin{align*}
	(a(t),b(t)): [0,1] \longrightarrow (0,\infty) \times (0,\infty): t \longmapsto (a_0 + t (a_1 - a_0), b_0 + t (b_1 - b_0)).
\end{align*}
We know that for $\left(a_0,b_0,\frac{b_0}{a_0}\right) \in S_0$ as above the eigenvalues of $A_0(a_0,b_0)$ are all positive and $S_0$ does not intersect the critical set
\begin{align*}
	S_{crit} := \left\{(a,b,r^*) \in \R^3 \left| \, a > -1, b > 0, r^* \in I_r, b \leq C_{crit}\right.\right\}.
\end{align*}
Thus the eigenvalues remain positive for all $(a(t),b(t))$ with $t \in [0,1]$. \\
Now we consider the SSCs corresponding to $\cos(\alpha^*) > 0$ given by
\begin{align*}
	S_+ := \left\{(a,b,r^*) \in \R^3 \left| \, a > -1, b > 0, \frac{b}{r^*} > a\right.\right\}.
\end{align*}
Now let $(a_1,b_1,r_1) \in S_+$ be arbitrary and use $b_0 := b_1$ and $a_0 := \frac{b_1}{r_1}$ as a starting point. Then $(a_0,b_0,r_1) = \left(a_0,b_0,\frac{b_0}{a_0}\right) \in S_0$ and therefore the eigenvalues of $A_0(a_0,b_0)$ are positive. While decreasing $a_0$ to $a_1$ - which is equivalent to increasing $\cos(\alpha^*)$ from $0$ to some positive value - it is still not possible to intersect $S_{crit}$, since $S_{crit}$ only allows for $\cos(\alpha^*) < 0$. Hence the eigenvalues remain also positive for this variation. This especially covers all cases where $a \leq 0$. \\
Finally we want to cover all the cases that are left over. For this define the set of all surfaces with $\cos(\alpha^*) < 0$ as
\begin{align*}
	S_- := \left\{(a,b,r^*) \in \R^3 \left| \, a > 0, b > 0, \frac{b}{r^*} < a\right.\right\}
\end{align*}
and let $(a_2,b_2,r_2) \in S_-$ be given and satisfy
\begin{align*}
	b_2 > -\frac{r_2}{3} \sqrt{1 - \left(\frac{b_2}{r_2} - a_2\right)^2} \left(\frac{b_2}{r_2} - a_2\right).
\end{align*}
Again we try to find a path that connects $(a_2,b_2,r_2) \in S_-$ with a configuration, where we know that all eigenvalues are positive. We remark that due to $(a_2,b_2,r_2) \in S_-$ we know that $r_2 > \frac{b_2}{a_2}$. Decreasing $r_2$ to $\frac{b_2}{a_2}$ brings us to a configuration in $S_0$, where we have only positive eigenvalues. During this decreasing process it is not possible that
\begin{align*}
	b_2 > -\frac{r_2}{3} \sqrt{1 - \left(\frac{b_2}{r_2} - a_2\right)^2} \left(\frac{b_2}{r_2} - a_2\right) = \frac{1}{3} \sqrt{1 - \left(\frac{b_2}{r_2} - a_2\right)^2} (a_2 r_2 - b_2)
\end{align*}
gets violated, since $\sqrt{1 - \left(\frac{b_2}{r_2} - a_2\right)^2} \geq 0$ and $a_2 r_2 - b_2$ is decreasing with $r_2$. This shows that the positivity of the eigenvalues is also valid for $(a_2,b_2,r_2)$. Hence assumption (d) of Theorem \ref{thm:GPLStability} is satisfied for all SSCs and parameters $(a,b) \in (-1,\infty) \times (0,\infty)$ that satisfy $b > C_{crit}$.

After we checked all assumptions required for the GPLS, we finally apply Theorem \ref{thm:GPLStability} and obtain the main result of this paper.

\begin{thm}[Stability of spherical caps]\label{thm:StabilitySCs}
Let $a > -1$, $b > 0$ and $4 < p < \infty$. Moreover, assume $\Gamma^*$ to be a stationary spherical cap with radius $R^*$ and contact angle $\alpha^*$ that satisfies $b > -\frac{1}{3} R^* \sin(\alpha^*)^2 \cos(\alpha^*)$. Then $\rho \equiv 0$ is stable in
\begin{align*}
	\tilde{X} := \left\{\rho \in W^{2-\frac{2}{p}}_p(\Gamma^*) \left| \rho|_{\d \Gamma^*} \in W^{3-\frac{3}{p}}_p(\d \Gamma^*)\right.\right\}
\end{align*}
and there exists $\delta > 0$ such that the unique solution $\rho(t)$ of the system (\ref{eq:Flow1})-(\ref{eq:Flow2}) with initial value $\rho_0 \in \tilde{X}$ satisfying $\norm{\rho_0}_{W^{2-\frac{2}{p}}_p(\Gamma^*)} + \norm{\rho_0|_{\d \Gamma^*}}_{W^{3-\frac{3}{p}}_p(\d \Gamma^*)} < \delta$ exists on $\R^+$ and converges at an exponential rate to some $\rho_\infty$, which parametrizes a stationary spherical cap as well.
\end{thm}
\begin{proof}
Reformulating the statement of Theorem \ref{thm:GPLStability} to the specific case of SCs as presented in this section.
\end{proof}

\printbibliography												

\end{document}